\documentclass{amsart}

\newtheorem{thm}{Theorem}[section]
\newtheorem{cor}[thm]{Corollary}
\newtheorem{lem}[thm]{Lemma}
\newtheorem{prop}[thm]{Proposition}
\theoremstyle{definition}
\newtheorem{defn}[thm]{Definition}
\newtheorem{assumption}[thm]{Assumption}
\theoremstyle{remark}
\newtheorem{rem}[thm]{Remark}
\numberwithin{equation}{section}
\newcommand{\Bhat}{\hat{B}}

\begin{document}

\title[]
 {Existence, Lifespan and Transfer Rate of Ricci Flows on Manifolds with Small Ricci Curvature}

\author{Fei He}

\address{School of Mathematics, University of Minnesota, Twin Cities}

\email{hexxx221@umn.edu}

\thanks{}

\thanks{}

\subjclass{}

\keywords{}

\date{}

\dedicatory{}

\commby{}

% -----------------------------------------------------------------------------

\begin{abstract}
We show that in dimension 4 and above, the lifespan of Ricci flows depends on the relative smallness of the Ricci curvature compared to the Riemann curvature on the initial manifold. We can generalize this lifespan estimate to the local Ricci flow, using which we prove the short-time existence of Ricci flow solutions on noncompact Riemannian manifolds with at most quadratic curvature growth, whose Ricci curvature and its first two derivatives are sufficiently small in regions where the Riemann curvature is large. Those Ricci flow solutions may have unbounded curvature. Moreover, our method implies that, under some appropriate assumptions, the spatial transfer rate (the rate at which high curvature regions affect low curvature regions) of the Ricci flow resembles that of the heat equation.
\end{abstract}

% -----------------------------------------------------------------------------
\maketitle
% -----------------------------------------------------------------------------

\section*{Introduction}
In the analysis of various parabolic geometric flows, estimating the lifespan is a fundamental problem which is closely related to the existence and regularity theories. The well-known Ricci flow equation
\begin{equation}\label{equation: ricci flow}
\frac{\partial}{\partial t} g = -2 Ric
\end{equation}
has been extensively studied. The initial value problem of Ricci flow can be solved on complete Riemannian manifolds with bounded curvature, and the solution can extend to larger time intervals as long as the curvature remains bounded. Therefore, to estimate the lifespan of the Ricci flow, one only need to control the curvature. The Riemann curvature tensor $Rm$ evolves by a second order parabolic equation under the Ricci flow. When the underlying manifold is compact, the maximum principle implies the so called doubling-time estimate (Lemma \ref{lemma: doubling time estimate of Hamilton}), which states that the spatial supremum of $|Rm|$, where $|\cdot|$ denotes the norm, will not double in a short time interval, whose length is reciprocal to the supremum of $|Rm|$ on the initial manifold.

The doubling-time lower bound gives an effective estimate of the lifespan of Ricci flows. It is actually sharp in form for the shrinking sphere example: $g(t) = (1- 2(n-1)t) g_{ \mathbb{S}^n}$ which becomes a point at $t=\frac{1}{2(n-1)}$. However, this estimate is not sharp when the Ricci curvature is small relative to $Rm$. Heuristically, if the initial manifold has vanishing Ricci curvature, then the Ricci flow always has a trivial long time solution, no matter how large the curvature is. Starting from a sufficiently small perturbation of the Ricci-flat metric, we shall still expect the solution to have a relatively long lifespan, which the doubling time estimate fails to predict. To deal with this situation, we prove the following version of doubling-time estimate, which features the role played by the Ricci curvature.

\begin{thm}\label{theorem: improved doubling time estimate}
Let $(M,g)$ be a compact Riemannian manifold with dimension $n$. Suppose there are constants $K >0$ and $A \geq 0$, such that $\sup_M |Rm| \leq K$, and $\sup_M |\nabla^m Ric | \leq (n-1) K^{1+\frac{m}{2}} e^{-AK}$ for $m=0,1,2$. Let $T_d$ be the first time that $\sup_M|Rm|$ becomes $2K$ under the Ricci flow starting from $(M,g)$, then
\[T_d \geq c(n)\max \{ A, \frac{1}{K}\},\]
where $c(n)$ is a constant depending only on $n$.
\end{thm}
\begin{rem}
Since the Ricci curvature controls $Rm$ in dimension $2$ and $3$, the above theorem is new only in dimension $n \geq 4$.
\end{rem}
When the Ricci curvature is very small relative to $Rm$, we can take $A$ large and Theorem \ref{theorem: improved doubling time estimate} implies a long lifespan. In particular on Ricci-flat manifolds we can take $A=\infty$, hence $T_d =\infty$. On the other hand, when the Ricci curvature dominates we can still make the assumption true by choosing $A=0$ and $K$ large enough. We prove Theorem \ref{theorem: improved doubling time estimate} in section \ref{subsection: uniformly small ricci curvature}. The key idea is to control the temporal growth of $|Rm|$ at each point directly by the second derivative of $Ric$. This idea has also been implemented in a different way in the recent work of Kotschwar, Munteanu and Wang \cite{kotschwar2015local}, where they established an explicit upper bound of $|Rm|$ in terms of its initial value, a uniform bound of $|Ric|$ in a space-time neighbourhood, and the elapsed time. However, their estimate loses sharpness when the Ricci curvature is small, hence it does not serve our purpose. Curvature and derivative estimates based on the space-time bound of $|Ric|$ have also been obtained by Chen in \cite{chen2016compactness}, where an injectivity lower bound or a pinching condition was also required on a time interval. Our results in this paper only need assumptions on the initial manifolds.

To make this lifespan estimate more useful, we shall restrict the smallness assumption on $Ric$ and its covariant derivatives only to regions where $Rm$ is large. In section \ref{subsection: non-uniforl smallness of ricci}, we will show lifespan estimates (Proposition \ref{theorem: lifespan estimate - compact case}, \ref{theorem: lifespan estimate - compact case - slow curvature growth}) for Ricci flows starting from compact Riemannian manifolds satisfying some good covering conditions, which generalize Theorem \ref{theorem: improved doubling time estimate}.

With some more efforts we can estimate the lifespan of the local Ricci flow (\ref{equation: local Ricci flow}) introduced by Deane Yang \cite{yang1992convergence}, see Proposition \ref{theorem: lifespan estimate - local ricci flow} in section \ref{subsection: lifespan estimate for the local ricci flow}. This estimate serves as a key ingredient to establish the following existence result for the Ricci flow on complete manifolds with potentially unbounded curvature, which will be proved in section \ref{subsection: existence of rf}.
\begin{defn}\label{definition: second order curvature scale}
For any point $x$ in a Riemannian manifold, we define $\rho_x$ to be the largest number such that
\[\sup_{B(x,\rho_x)} |Rm| \leq \frac{1}{\rho_x^{2}},\]
\[\sup_{B(x,\rho_x)} |\nabla^m Ric|(x) \leq \frac{n-1}{\rho_x^{2+m}}, \quad m=1,2.\]
\end{defn}
Clearly $\rho_x$ always exists and $\rho_x>0$ for any point $x$ in a smooth Riemannian manifold.
\begin{thm}\label{theorem: existence of ricci flow strong assumption}
Let $(M,g)$ be a compete Riemannian manifold with dimension $n$. Let $\rho_x$ be as in Definition \ref{definition: second order curvature scale}. Suppose there are constants $\alpha, \beta, \gamma>0$, such that for any $x, y \in M$, we have

(i) $|\nabla^m Ric|(x) \leq (n-1)\rho_x^{-2-m} e^{\beta- \alpha\rho_x^{-2}}$ for $m=0,1,2$,

(ii) $|\rho_x^{-2} - \rho_y^{-2}|\leq \gamma$ when $d(x,y) < \rho_x + \rho_y$.

\noindent Then there exists a complete solution of the Ricci flow $g(t)$, $t\in [0,T_0]$, with $g(0)=g$, where $T_0>0$ depends only on $n, \alpha, \beta, \gamma$. Moreover, $(M,g(t))$ is equivalent to $(M,g)$ and has at most quadratic curvature growth.
\end{thm}
\begin{rem}\label{remark: unbounded curvature}
Assumption (ii) implies that the initial manifolds in Theorem \ref{theorem: existence of ricci flow strong assumption} can have at most quadratic curvature growth, in that case, so will the solutions for a short time.
\end{rem}
\begin{rem}
The condition $|\nabla^m Ric|(x) \leq (n-1)\rho_x^{-2-m} e^{\beta- \alpha\rho_x^{-2}}$ for $m=0,1,2$, is trivial when $\rho_x \geq \sqrt{\alpha/\beta}$. Therefore the smallness of the Ricci curvature is only assumed on highly curved regions.
\end{rem}

Existence of Ricci flow solutions on noncompact manifolds was first proved by Shi in \cite{shi1989deforming}, where the curvature is assumed to be uniformly bounded. Under weaker curvature conditions and some noncollapsing assumptions, short-time existence has been proved by Xu \cite{xu2013short}, Simon \cite{simon2012ricci} and Hochard \cite{hochard2016short}, where the curvature may be unbounded initially, but becomes uniformly bounded instantaneously for any positive time. See also the existence result of Chau, Li and Tam in \cite{chau2014deforming} for the K\"ahler Ricci flow. Solutions with unbounded curvature for positive time have been obtained by Cabezas-Rivas and Wilking in \cite{cabezas2011produce}, where the manifolds are assumed to have nonnegative complex sectional curvature. And by Giesen and Topping in \cite{giesen2013ricci} for 2 dimensional manifolds. In general, the existence of Ricci flow solutions with unbounded curvature is a very subtle problem, one can refer to \cite{topping2014ricci} for a survey. On one hand the existence is expected to be not true in general, a possible counter-example is a sequence of spheres connected by thinner and thinner necks, which should be pinched immediately under the flow. On the other hand, since the Ricci flow is such a natural geometric equation, it'll be interesting to know what initial manifolds admit a short-time solution.

The solutions provided by Theorem \ref{theorem: existence of ricci flow strong assumption} may also have unbounded curvature for positive time, as pointed out by Remark \ref{remark: unbounded curvature}.  We would also like to point out that the uniqueness of Ricci flows with quadratic curvature growth and some other mild assumptions has been established by Kotschwar in \cite{kotschwar2014energy}. It's easy to check that our solutions satisfy these conditions.

In section \ref{subsection: existence of rf} we will prove the existence of Ricci flows on manifolds satisfying a good covering condition (\ref{initial good cover assumption}), see Theorem \ref{theorem: existence of ricci flow on noncompact manifolds}. Our approach is to take limit of a sequence of local Ricci flows with domains exhausting the entire manifold, this strategy has also been applied in \cite{xu2013short} and \cite{yang1992convergence}. The main difficulty in the analysis of the local Ricci flow is that higher order derivatives of the cut-off function appears in evolution equations of the curvatures, and need to be controlled. In section \ref{section: derivatives estimates for the local ricci flow} we establish some derivative estimates for the local Ricci flow, which may be of independent interests.

During the proof of lifespan estimates we establish barrier function upper bounds for the norm of the Ricci curvature and its first two derivatives, which also imply an estimate for the spatial `transfer rate' -- a term borrowed from heat conduction -- of the Ricci flow. Recall that the heat transfer rate reflects how fast (or slow) lower temperature regions are affected  by higher temperature regions. As implied by the fundamental solution of the heat equation on $\mathbb{R}^n$, it decays at the rate of $e^{-d^2}$, where $d$ is the distance to the origin. Since the Ricci curvature tensor can be viewed as the Laplacian of the metric tensor, it is reasonable to expect the following Ricci flow analogue, which is proved in section \ref{section: transfer rate of the ricci flow}.

%\begin{defn}(Curvature Scale at $x$)
%We define the curvature scale $\rho_x$ at $x$ to be the largest number such that
%\[\sup_{B(x,\rho_x)} |Rm| \leq \frac{1}{\rho_x^2}.\]
%\end{defn}
%
%\begin{defn}\label{definition: kappa noncollapsed under the curvature scale}
%We say a Riemannian manifold $M$ is $\kappa$-noncollapsed under the curvature scale if there is a constant $\kappa >0$, such that for any $x\in M$,
%\[ \frac{Vol B(x, r)}{r^n} \geq \kappa, \quad for \hspace{4pt} 0< r< \rho_x.\]
%\end{defn}

\begin{thm}\label{theorem: transfer rate estimate}
Suppose $(M,g)$ is a noncompact manifold satisfying
\[|Rm|(x) \leq (1+d(x,p))^{2}, \]
\[|\nabla^m Ric|(x) \leq (1+d(x,p))^{2+m} \exp ( -\alpha (1+d(x,p))^2), \quad m=0,1,2,\]
for all $x\in M \backslash \Omega_0$, where $\Omega_0$ is a compact subset, $p\in M$ is a fixed point, $\alpha$ is a positive constant and $d(\cdot,\cdot)$ is the distance function. Then there is a complete solution of the Ricci flow on $M\times [0,T_1]$ with initial metric $g$, whose Ricci curvature has decay:
\[|\nabla^m Ric|(x) \leq C_1 (1+d(x,p))^{2+m} \exp ( - C_2 (1+d(x,p))^2), \quad m=0,1,2,\]
for $(x,t) \in M\backslash \Omega_0 \times [0, T_1]$, where the constants $C_1, C_2$ and $T_1 >0$ depend on $\Omega_0, n$ and $ \alpha$.
\end{thm}

\begin{rem}
It is easy to see that the temporal growth of the geometry under the solution in Theorem \ref{theorem: transfer rate estimate} has exponential decay, more precisely,
\[ |\ln (\frac{g(x,t)}{g(x,0)})| \leq C_1 \exp ( - C_2 (1+d(x,p))^2),  \]
\[
| |Rm|(x,t)-|Rm|(x,0) | \leq C_1 (1+d(x,p))^2 \exp ( -C_2 (1+d(x,p))^2),
\]
for any $(x,t) \in M \backslash \Omega_0 \times [0,T_1]$ and some constants $C_1, C_2$.
\end{rem}
A interesting particular case is described by the following immediate corollary.
\begin{cor}
Suppose a complete Riemannian manifold $(M,g)$ is Ricci-flat outside of a compact domain $\Omega$, and has at most quadratic curvature growth. Then $(M,g)$ admits a short-time solution of the Ricci flow, whose Ricci curvature has spatial decay at the rate of $e^{-d^2}$, where $d$ is the distance to $\Omega$ with respect to $g$.
\end{cor}
Similar decay estimates for Ricci flows with bounded curvature coming out of Ricci-flat cones have been proved by Siepmann \cite{siepmann2013ricci} and Deruelle \cite{deruelle2014asymptotic}.

Finally we would like to point out that parallel results can be proved by the same method for the normalized Ricci flow
\[\frac{\partial}{\partial t} g = -2(Ric -\lambda g), \quad \lambda=Const,\]
where the smallness assumption should be imposed on the tensor $Ric -\lambda g$ instead of $Ric$.

%%%%%%%%%%%%%%%%%%%%%%%%%%%%%%%%%%%%%%%%%%%%%%%%%%%%%%%%%%%%%%%%%%%%%%%%%%%%%%%%%%%%%%%%%%%%%%%%%%%%%%%%%%%%%%%%%%%%%%%%%%%%%%%%%%%%%
\textbf{Acknowledgements:} The author is grateful to Prof. Jiaping Wang for very helpful conversations. He would also like to thank Prof. Robert Gulliver for his interest in this work.
%%%%%%%%%%%%%%%%%%%%%%%%%%%%%%%%%%%%%%%%%%%%%%%%%%%%%%%%%%%%%%%%%%%%%%%%%%%%%%%%%%%%%%%%%%%%%%%%%%%%%%%%%%%%%%%%%%%%%%%%%%%%%%%%%%%%%

\section{Implications of bounded $\nabla^2 Ric$}
By the second Bianchi identity and the commutation formula for covariant derivatives, we can calculate directly that
\begin{equation}\label{equation: elliptic equation for Rm}
\Delta R_{ijkl} = Q(Rm)_{ijkl} + \nabla_k\nabla_i R_{jl} - \nabla_k \nabla_j R_{il} +\nabla_l \nabla_j R_{ik} - \nabla_l \nabla_i R_{jk},
\end{equation}
where $Q(Rm)$ is a quadratic term. This equation allows us to estimate the first derivative of $Rm$ in terms of bounds on $|Rm|$ and $|\nabla^2 Ric|$. In the following, we use $B_g(p,r)$ to denote a geodesic ball with radius $r$, centered at a point $p$ on a complete $n$-dimensional Riemannian manifold $(M,g)$.

\begin{lem}\label{lemma: control nabla Rm by Rm and nabla nabla Ric}
Suppose $|Rm|(x)\leq C_0$ and $|\nabla^2 Ric|(x)\leq C_0$, for all $x\in B_g(p,1)$, then there is a constant $C_1$ depending on $C_0$ and the dimension $n$, such that $|\nabla Rm|(x)\leq C_1$ for any $x\in B_g(p,\frac{1}{4})$.
\end{lem}
\begin{proof}
Equation (\ref{equation: elliptic equation for Rm}) implies
\begin{equation}\label{equation: elliptic equation for nabla Rm}
\Delta \nabla Rm = Rm*\nabla Rm + \nabla^3 Ric,
\end{equation}
where $\nabla^3 Ric$ denotes a linear combination of $4$ terms of third order derivative of $Ric$. Here and in the following we use $*$ to denote a contraction by the Riemannian metric.

Let $\phi$ be a cut-off function compactly supported in $B_g(p,1/2)$ which equals $1$ in $B_g(p,1/4)$, and satisfies $|\nabla \phi|\leq 8$.
Multiply both sides of (\ref{equation: elliptic equation for nabla Rm}) by $\phi^2 |\nabla Rm|^{2p-2} \nabla Rm$ and integrate by parts to get
\[
\begin{split}
&\int \phi^2 |\nabla Rm|^{2p-2} |\nabla^2 Rm|^2 + (p-1) \int \phi^2 |\nabla Rm|^{2p-2} |\nabla |\nabla Rm||^2 \\
=& \int 2 \phi |\nabla Rm|^{2p-2} \nabla \phi * \nabla Rm * \nabla^2 Rm + \int \phi^2 |\nabla Rm|^{2p-2} Rm * \nabla Rm * \nabla Rm \\
& + \int \phi^2 |\nabla Rm|^{2p-2} \nabla Rm * \nabla^3 Ric.
\end{split}
\]
Integrate by parts again to get rid of $\nabla^3 Ric$ in the last term,
\[
\begin{split}
&\int \phi^2 |\nabla Rm|^{2p-2} \nabla Rm * \nabla^3 Ric \\
= & 2 \int \phi^2 |\nabla Rm|^{2p-2} \nabla^2 Rm * \nabla^2 Ric - \int 2\phi |\nabla Rm|^{2p-2} \nabla \phi * \nabla Rm * \nabla^2 Ric  \\
& - (2p-2)\int \phi^2 |\nabla Rm|^{2p-3} \nabla |\nabla Rm| * \nabla Rm * \nabla^2 Ric.
\end{split}
\]
Then the Cauchy-Schwarz inequality and standard absorbing technique imply that
\[
\begin{split}
&\int \phi^2 |\nabla Rm|^{2p-2} |\nabla^2 Rm|^2 + (p-1) \int \phi^2 |\nabla Rm|^{2p-2} |\nabla |\nabla Rm||^2 \\
\leq & C \int_{B_g(p,1)} |\nabla Rm|^{2p}  + Cp \int \phi^2 |\nabla Rm|^{2p-2},
\end{split}
\]
where $C$ is a constant depending on $n$ and $C_0$.
The Sobolev inequality (\ref{lemma: Sobolev inequality}) and the Nash-Moser iteration technique yield
\[
\sup_{B_g(p, 1/4)} |\nabla Rm|^2 \leq C(n,C_0)\left( 1+  \frac{1}{Vol_g(B_g(p,1))} \int_{B_g(p,1/2)} |\nabla Rm|^2 \right).
\]
Then the claimed result follows from Lemma \ref{lemma: l2 estimate of nabla Rm}.
\end{proof}

\begin{lem}\label{lemma: l2 estimate of nabla Rm}
Suppose $|Rm|(x)\leq C_0$ and $|\nabla^2 Ric|(x)\leq C_0$ for all $x\in B_g(p,1)$, then there is a constant $C_2$ depending on $C_0$ and the dimension $n$, such that $\int_{B_g(p,1/2)}|\nabla Rm|^2\leq C_2 Vol_g(B_g(p,1))$.
\end{lem}
\begin{proof}
Equation (\ref{equation: elliptic equation for Rm}) implies
\[
\Delta |Rm|^2= 2|\nabla Rm|^2 + 2\langle Rm, Rm* Rm + \nabla^2 Ric\rangle.
\]
For any cut-off function $\phi$ compactly supported in $B_g(p,1)$,
\[
\begin{split}
2\int \phi^2 |\nabla Rm|^2 \leq & \int \phi^2 \Delta |Rm|^2 + C(n) \int \phi^2 (|Rm|^3 + |Rm||\nabla^2 Ric|). \\
\end{split}
\]
Integration by parts yields
\[
\begin{split}
\int \phi^2 \Delta |Rm|^2 = & -4\int \phi |Rm|\langle\nabla \phi, \nabla |Rm|\rangle \\
\leq & \int \phi^2 |\nabla Rm|^2 + 4\int |\nabla \phi|^2|Rm|^2.
\end{split}
\]
Then the proof is finished by choosing $\phi=1$ on $B_g(p,1/2)$ and $|\nabla \phi|\leq 2$.
\end{proof}

The Sobolev inequality used in the Nash-Moser iteration is provided by the well-known result due to Saloff-Coste \cite{saloff1992uniformly}, which we state here for reference in the following sections.
\begin{lem}[Saloff-Coste]\label{lemma: Sobolev inequality}
Suppose $n>2$ and $Ric \geq -(n-1)Kg$ on $B_g(p,r)$, where $K>0$ is a constant. Then there is a constant C depending only on $n$, such that
\begin{equation}\label{equation: sobolev inequality by Saloff-Coste}
\left( \int |f|^\frac{2n}{n-2} \right)^\frac{n-2}{n} \leq e^{C(1+\sqrt{K}r)}\frac{r^2}{Vol_g(B_g(p,r))^{2/n}} \int (|\nabla f|^2 + r^{-2} f^2),
\end{equation}
for any $W^{1,2}$ function $f$ compactly supported in $B_g(p,r)$.
\end{lem}

%Also note that the Ricci lower bound implies a uniform volume doubling constant for geodesic balls with uniformly bounded radius, which is a consequence of the volume comparison theorem. This simple fact is used in the proof in Lemma \ref{lemma: control nabla Rm by Rm and nabla nabla Ric}, and will be used repeatedly in the following sections.

In the following sections of this paper, we need cut-off functions with bounded higher order (up to three) derivatives. A key ingredient is the smooth distance-like function with higher derivative bounds constructed by Tam \cite{tam}, see also \cite{chow2010ricci}.
\begin{lem}[L.-F. Tam]\label{lemma: distance like function with higher derivative bounds}
Let $B(p,r)$ be a geodesic ball properly embedded in an $n$-dimensional complete Riemannian manifold $(M,g)$, suppose $r\geq 1$. There exist constants $C_0$ and $C_1$ depending only on $n$, such that if $\sup_{B(p,r)}|Rm|\leq 1$ and $\sup_{B(p,r)}|\nabla Rm|\leq 1$, then there is a function $f: M \to \mathbb{R}$ with the following bounds on $B(p,2r/3)$:

(i) $d_g(x,p)+1 \leq f(x) \leq d_g(x,p)+C_0$;

(ii) $\sup_M|\nabla^m f(x)| \leq C_1$, $m=1,2,3$.
\end{lem}
\begin{proof}
The idea is to evolve a smooth distance-like function constructed in \cite{greene1979c} by the heat equation. We refer to Proposition 26.49 and Remark 26.50 in \cite{chow2010ricci} for a complete proof when $r=\infty$. Note that the heat kernel upper bound used in this proof can be obtained locally, and the bounds on $|Rm|$ and $|\nabla Rm|$ are also used locally to estimate higher derivatives of $f$. Therefore we can simply cut off the initial data on $B(p,r)$ and convolute with a Dirichlet heat kernel, then the same proof as in \cite{chow2010ricci} yields the lemma.
\end{proof}

The desired cut-off function is constructed in the following lemma.
\begin{lem}\label{lemma: cut-off function on each ball}
Let $B(r)$ be a geodesic ball with radius $r$ in a complete Riemannian manifold $(M,g)$ with dimension $n$. Suppose $\sup_{B(r)} |Rm|\leq K$, $\sup_{B(r)} |\nabla Rm| \leq K^{3/2}$, and $r\geq \bar{r}/\sqrt{K}$, where $\bar{r}=4C_0$ for the same constant $C_0$ in Lemma \ref{lemma: distance like function with higher derivative bounds}. Then there is a smooth function $0\leq \phi\leq 1$ compactly supported in $B(r)$, with $\phi(x)\equiv 1$ on $B(r/4)$, and $\sup_{B(r)}|\nabla^m \phi|\leq C K^{m/2}$, $m=1,2,3$, where $C$ depends only on $n$.
\end{lem}
\begin{proof}
Let $0\leq \psi(s) \leq 1$ be a smooth function on $\mathbb{R}$ such that $\psi(s)=1$ when $s\leq \frac{1}{4}$, $\psi(s)=0$ when $s>\frac{1}{3}$, and $|\psi'(s)|+|\psi''(s)|+|\psi'''(s)|\leq 100$ for all $s$.
Let $f$ be the smooth distance-like function constructed in Lemma \ref{lemma: distance like function with higher derivative bounds} on $B(r)$ with the rescaled metric $\tilde{g}=Kg$. Define $\phi(x)=\psi((\sqrt{K}r+\bar{r})^{-1}f(x))$. One can easily check that $\phi$ is compactly supported on $B(r)$ and $\phi(x)=1$ for any $x\in B(r/4)$, and
\[|\nabla_{\tilde{g}}^m \phi(x)|_{\tilde{g}}\leq C(n), \quad m=1,2,3,\quad x\in B(r).\]
By the scaling property of covariant derivatives, we have
\[|\nabla_{g}^m \phi(x)|_{g}\leq C(n)K^{m/2}, \quad m=1,2,3,\quad x\in B(r).\]
%when $\rho(x)\leq \sqrt{K}r/4$, $f(x)\leq \rho(x)+C_0\leq \sqrt{K}r/4+C_0$
%
%when $\rho(x)> \sqrt{K}r$, $f(x)> \sqrt{K}r+1$
%\[f(x)> \frac{\sqrt{K}r+1}{\sqrt{K}r+\bar{r}}\]
\end{proof}
\section{Lifespan of Ricci flows on compact manifolds}\label{section: lifespan cpt case}

\subsection{Uniformly small Ricci curvature}\label{subsection: uniformly small ricci curvature}
Under the Ricci flow, the Riemann curvature tensor evolves by a parabolic equation
\begin{equation}
\frac{\partial }{\partial t} Rm = \Delta Rm + Rm * Rm.
\end{equation}
When the manifold is compact, the maximum principle implies the doubling-time estimate which plays a fundamental role in the analysis of the Ricci flow:
\begin{lem}[Hamilton \cite{hamilton1982three}]\label{lemma: doubling time estimate of Hamilton}
Let $g(t)$ be a Ricci flow solution on a compact manifold $M$ with dimension $n$. Let $T_d$ be the first time that $\sup_{M \times [0,T_d)} |Rm| =2 \sup_M |Rm| (0)$, then $T_d \geq \frac{c(n)}{\sup_M |Rm| (0)}$, where $c(n)$ is a constant depending only on the dimension.
\end{lem}

Theorem \ref{theorem: improved doubling time estimate} can be viewed as an improved version of the above result of Hamilton. Although it can be implied by the more general Proposition \ref{theorem: lifespan estimate - compact case - slow curvature growth} in the following subsection, we would like to present its proof here as a more economic illustration of how the smallness of the Ricci curvature is used.

Basic ingredients in the proof are the well-known Shi's estimate, and it's modified version which is stated here as a lemma.
\begin{lem}[Lu-Tian \cite{lu2005uniqueness}]\label{lemma: modified shi's estimates}
For any $\alpha, K>0$, and any integers $ n \geq 2, l \geq 0, m \geq 1$, there exists a constant $C$ depending only on $\alpha, n, l$ and $m$ such that if $M$ is a manifold of dimension $n$, $p \in M$, and $g(t), t\in [0,T]$, $0< T < \frac{\alpha }{K}$, is a solution to the Ricci flow on an open neighbourhood $\mathcal{U}$ containing $\bar{B}_{g(0)}(p,r)$, $0< r < \frac{\alpha}{\sqrt{K}}$, and if
\[|Rm|(x,t) \leq K \quad for \hspace{2pt} all \hspace{2pt} (x,t) \in \mathcal{U} \times [0,T],
\]
and
\[|\nabla ^i Rm|(x,0) \leq K^{1+\frac{i}{2}} \quad for \hspace{2pt} all \hspace{2pt} x \in \mathcal{U}, \hspace{2pt} and \hspace{2pt} 1 \leq i \leq l,\]
then
\[ | \nabla^m Rm|(y,t) \leq \frac{CK^{1+\min \{ m, l\}/2}}{t^{(m-l)_+/2}}\]
for all $(y,t) \in B_{g(0)}(p, \frac{r}{2}) \times (0, T]$.
\end{lem}
Recall that in Theorem \ref{theorem: improved doubling time estimate} we have uniform curvature assumptions that $\sup_M|Rm|\leq K$ and $|\nabla^m Ric|\leq K^{1+m/2}e^{-AK}$ for $m=0,1,2$, where $A\geq 0$. And we use $T_d$ to denote the first time that $\sup_M|Rm|(t)$ becomes $2K$.

\begin{proof}[Proof of Theorem \ref{theorem: improved doubling time estimate}]
We proof this theorem in three steps.

\noindent Step 1: Control the first two derivatives of $Rm$.

By Lemma \ref{lemma: control nabla Rm by Rm and nabla nabla Ric}, $\sup_M|\nabla Rm|(0)\leq C(n)K^{3/2}$.

For $t < 1/K$, and any $x\in M$, the modified Shi's estimate \ref{lemma: modified shi's estimates} implies that
\[|\nabla Rm|(x,t)\leq C(n)K^{3/2},\]
and
\[|\nabla^2 Rm|(x,t)\leq \frac{C(n)K^{3/2}}{t^{1/2}}.\]
For $1/K \leq t \leq T_d$, we can apply Shi's estimate on the time interval $(t-1/K, t)$ to get
\[|\nabla Rm|(x,t)\leq C(n)K^{3/2},\]
and
\[ |\nabla^2 Rm|(x,t)\leq C(n)K^2,\]
for any $x\in M$.
%Therefore we have
%\[|\nabla Rm|(x,t)\leq C(n)K^{3/2},\quad (x,t) \in M\times [0,T_d],\]
%\[|\nabla^2 Rm|(x,t)\leq \frac{C(n)K^{3/2}}{t^{1/2}}, \quad (x,t)\in M\times [0, 1/K],\]
%and
%\[|\nabla^2 Rm|(x,t)\leq C(n)K^2, \quad (x,t)\in M\times [1/K, T_d].\]

\noindent Step 2: Control $Ric$ and its first two derivatives.

Along the Ricci flow, the Ricci curvature evolves by
\begin{equation}\label{equation: evolution equation of Ric along the Ricci flow}
\frac{\partial}{\partial t} Ric =\Delta Ric + Rm * Ric.
\end{equation}
Hence
\[\frac{\partial}{\partial t}|Ric| \leq \Delta |Ric| + C_1K |Ric|,\]
as long as $t \leq T_d$, where $C_1$ depends only on $n$. By the maximum principle we can estimate
\[\sup_M|Ric|(t) \leq \sup_M|Ric|(0) e^{C_1 K t} \leq Ke^{C_1Kt-AK}.\]

The first derivative of the Ricci curvature evolves by
\begin{equation}\label{equation: evolution equation of the first covariant derivative of the Ricci curvature}
\frac{\partial}{\partial t} \nabla Ric =\Delta \nabla Ric + Rm * \nabla Ric + \nabla Rm * Ric.
\end{equation}
For $t\leq T_d$ we have
\[\frac{\partial}{\partial t}|\nabla Ric|\leq \Delta |\nabla Ric| + C_2 K |\nabla Ric| + C_2 K^{5/2}e^{C_1Kt-AK},\]
where $C_2$ depends only on $n$. The maximum principle implies
\[
\begin{split}
\sup_M|\nabla Ric|(t)\leq & (\sup_M|\nabla Ric|(0)+ C_2C_1^{-1} K^{3/2}e^{C_1Kt-AK}) e^{C_2 K t}\\
\leq & C_3 K^{3/2}e^{C_3 K t - AK},
\end{split}
\]
for some $C_3$ depending only on $n$.

The second derivative of the Ricci curvature evolves by
\begin{equation}\label{equation: evolution equation of the second covariant derivative of the Ricci curvature}
\frac{\partial}{\partial t} \nabla^2 Ric =\Delta \nabla^2 Ric + Rm * \nabla^2 Ric + \nabla Rm * \nabla Ric + \nabla ^2 Rm * Ric.
\end{equation}
For $t\leq 1/K$, we have, for a constant $C_4$ depending only on $n$, that
\[
\begin{split}
\frac{\partial}{\partial t}|\nabla ^2 Ric|\leq &\Delta |\nabla^2 Ric| + C_4 K |\nabla^2 Ric|+ C_4 K^{3}e^{C_3Kt-AK} \\
&+ \frac{C_4K^{5/2}}{t^{1/2}} e^{C_1 K t -AK}.
\end{split}
\]
Since
\[\int_0^{1/K} \frac{C_4K^{5/2}}{t^{1/2}} e^{C_1 K t -AK} dt \leq 2C_4 e^{C_1} K^2 e^{-AK},\]
we can use the maximum principle as before to show that
\[\sup_M |\nabla^2 Ric|(t) \leq C_5 K^2 e^{C_5Kt-AK},\]
for some constant $C_5$ depending only on $n$.

For $1/K \leq t \leq T_d$ we have
\[
\begin{split}
\frac{\partial}{\partial t}|\nabla ^2 Ric|\leq \Delta |\nabla^2 Ric| + C_4 K |\nabla^2 Ric|+ C_4 K^{3}e^{C_3Kt-AK}+ C_4K^3 e^{C_1 K t -AK}.
\end{split}
\]
Apply the maximum principle with the initial value
\[\sup_M |\nabla^2 Ric|(1/K) \leq C_5 K^2 e^{C_5-AK},\]
we can find a constant $C_6$ depending only on $n$, such that
\[\sup_M|\nabla^2 Ric|(t) \leq C_6 K^2 e^{C_6 Kt -AK}, \quad t\in[1/K, T_d].\]
Therefore we can take $C_7=\max\{C_5, C_6\}$ to get a uniform estimate of $|\nabla^2 Ric|$,
\[\sup_M|\nabla^2 Ric|(t) \leq C_7 K^2 e^{C_7 Kt -AK}, \quad t\in[0, T_d].\]

\noindent Step 3: Control $Rm$ by $Ric$ and $\nabla^2 Ric$

Under the Ricci flow, the Riemman curvature tensor evolves by
\begin{equation}\label{equation: time derivative of Rm in terms of the second derivative of Ric}
\begin{split}
\frac{\partial}{\partial t} R_{ijk}^l =& -\nabla_i\nabla_k R_j^l + \nabla_i\nabla^l R_{jk} +\nabla_j\nabla_k R_i^l -\nabla_j\nabla^l R_{ik} \\
& + g^{lp} ( R_{ijk}^q R_{pq} + R_{ijp}^q R_{kq}),
\end{split}
\end{equation}
which implies
\[\frac{\partial}{\partial t} |Rm| \leq 6 |Rm||Ric| + 4 |\nabla^2 Ric|\]
at any point $x\in M$. Hence by Gronwall's inequality we have a pointwise estimate
\[
\begin{split}
|Rm|(t) \leq &\left(|Rm|(0)+ 4\int_0^t |\nabla^2 Ric|\right)\exp({6\int_0^t |Ric|}). \\
\end{split}
\]
By the estimates in Step 2, we have
\[\int_0^t |Ric| \leq e^{-AK}(e^{C_1Kt}-1),\]
and
\[\int_0^t |\nabla^2 Ric| \leq Ke^{-AK}(e^{C_7 Kt}-1).\]
Hence
\[\sup_M |Rm|(t) \leq K (1+4e^{-AK}(e^{C_7 Kt}-1))\exp(6e^{-AK}(e^{C_1Kt}-1)).\]
Let $\alpha = \min \{ \frac{\ln (4/3)}{6}, \frac{1}{8}\}$, then the following Lemma \ref{lemma: minimum of a continuous function} implies that for any $t < \frac{\alpha A}{(\alpha +1)(C_1+C_7)}$ which is independent of $K$, we have
\[t < \frac{\ln (\alpha e^{A K} + 1)}{ (C_1+C_7)K},\]
hence
\[e^{-AK}(e^{C_1Kt}-1) < \frac{\ln (4/3)}{6},\]
and
\[e^{-AK}(e^{C_7 Kt}-1) < \frac{1}{8},\]
then
\[\sup_M |Rm|(t) < K (1+ \frac{1}{2}) \frac{4}{3} = 2K.\]
Therefore we must have $T_d \geq \frac{\alpha A}{(\alpha +1)(C_1+C_7)}$. Together with Lemma \ref{lemma: doubling time estimate of Hamilton}, this lower bound yields the claimed result in the theorem.
\end{proof}

It is elementary to prove the following:
\begin{lem}\label{lemma: minimum of a continuous function}
For positive constants $\alpha, \beta > 0$, let
\[h(\alpha, \beta)(s) = \frac{\ln (\alpha e^{\beta s} + 1)}{ s}, \quad s \in (0, \infty).
\]
Then
\[\inf_{s \in (0, \infty)} h(\alpha, \beta)(s) \geq \frac{\alpha \beta}{\alpha + 1}.\]
\end{lem}

\subsection{Non-uniform smallness of the Ricci curvature}\label{subsection: non-uniforl smallness of ricci}
The uniform smallness assumption on the Ricci curvature in Theorem \ref{theorem: improved doubling time estimate} is very restrictive, for example, it would force a complete Riemannian manifold with unbounded curvature to be Ricci-flat. Fortunately, we can still obtain a lower estimate of the lifespan of Ricci flows when their initial data satisfies the following assumption, where the Ricci curvature and its first two derivatives are assumed to be small only in regions where $|Rm|$ is large. This assumption can be implied by the conditions in Theorem \ref{theorem: existence of ricci flow strong assumption} and \ref{theorem: transfer rate estimate}. Though lengthy, it is actually convenient to use in our arguments.
\begin{assumption}\label{initial good cover assumption}
(a) $M$ is covered by geodesic balls $B_i=B_g(x_i, r_i)$, $i=1, 2, 3, ... N$, where $N$ is finite when $M$ is compact and $N=\infty$ when $M$ is noncompact. We denote $\Bhat_i= B_g(x_i, 16r_i)$.

(b) Let $\bar{r}=4C_0$, where $C_0$ is the same constant as in Lemma \ref{lemma: distance like function with higher derivative bounds}, which depends only on $n$. For each $\Bhat_i$, there is a positive number $K_i \geq \max \{\bar{r}^2/r_i^2, 1\}$, such that \\ $\sup_{\Bhat_i} |Rm| \leq K_i$.

(c) There are positive constants $A$ and $\bar{K}$, such that for each $\Bhat_i$, $\sup_{\Bhat_i} |\nabla ^m Ric| \leq K_i^{m/2}P_i$, $m=0,1,2$, where $P_i=(n-1) K_i e^{A(\bar{K}-K_i)}$.

(d) There is a positive constant $\Gamma$, such that if $\Bhat_j \cap \Bhat_i \neq \emptyset$, then $K_j \leq K_i + \Gamma$.

(e) The number of $\Bhat_i$'s intersecting simultaneously is at most $I < \infty$.

%(f) There is a constant $\kappa>0$, such that for each $x\in \Bhat_i$, and any $0< r \leq \frac{1}{\sqrt{K_i}}$, we have
%\[
%\frac{Vol B(x,r)}{r^n} \geq \kappa.
%\]
\end{assumption}

A few remarks are at hand:
\begin{rem}
Assumption (c) is trivial if $\sup_{\Bhat_i} |Rm|\leq \bar{K}$ and $\sup_{\Bhat_i}|\nabla^m Ric|\leq (n-1)\bar{K}^{1+m/2}$, $m=0,1,2$, in which case we can take $K_i=\max\{\bar{K}, 1\}$.
\end{rem}
\begin{rem}
Consider a chain of connected balls $\Bhat_i$, $i=1,2,...,k,...$, the curvature upper bound grows linearly in $k$ along this chain, while the distance grows roughly as the partial sum $\sum_{i=1}^k \frac{1}{\sqrt{i}}$, which diverges at the rate of $\sqrt{k}$. Hence $|Rm|$ is allowed to have quadratic growth when $M$ is noncompact.
\end{rem}
\begin{rem}\label{remark: uniformly bounded ricci curvature}
Under Assumption \ref{initial good cover assumption}, the manifold actually has uniformly bounded Ricci curvature $|Ric|\leq (n-1) A^{-1}e^{A\bar{K}-1}$.
\end{rem}
\begin{prop}\label{theorem: lifespan estimate - compact case}
Let $(M,g)$ be an $n$-dimensional compact Riemannian manifold satisfying Assumption \ref{initial good cover assumption}. Then there is a constant $T_0$ depending only on $n, I, \Gamma, A$ and $\bar{K}$, such that the lifespan of the Ricci flow solution on $M$ with initial metric $g$ is greater than $T_0$. Moreover,
\[\sup_{\Bhat_i} |Rm|_{g(t)} \leq 2(1+\Gamma)K_i \hspace{4pt} for \hspace{4pt} any \hspace{4pt} t\in [0,T_0] \hspace{4pt} and \hspace{4pt}i=1,2,..., N.\]
\end{prop}
The proof of Proposition \ref{theorem: lifespan estimate - compact case} is omitted since it can be viewed as a corollary of Proposition \ref{theorem: lifespan estimate - local ricci flow}, which states the same result for the local Ricci flow. Indeed, we only need to take the cut-off function $\chi\equiv 1$ in Proposition \ref{theorem: lifespan estimate - local ricci flow}. We would like to point out that in order to control the Ricci curvature, we need to construct barrier functions by gluing up local information on each $\Bhat_i$, thus we need assumptions (d) and (e) to control the growth of the geometry.

From the proof of Proposition \ref{theorem: lifespan estimate - local ricci flow}, we see that the constant $T_0$ in Proposition \ref{theorem: lifespan estimate - compact case} can be written as $C(n, I, \Gamma) e^{-A\bar{K}-7A\Gamma}A$. However, larger constant $A$ in Assumption \ref{initial good cover assumption} implies smaller Ricci curvature, hence we should expect longer lifespan of the solutions. Therefore the exponential term $e^{-A\bar{K}-7A\Gamma}$ is undesirable. The next proposition shows that it can be avoided under slightly stronger assumptions:
 \begin{assumption}\label{initial good cover assumption -  slow curvature growth}
Instead of (c) and (d) in Assumption \ref{initial good cover assumption}, we assume

(c$'$) For each $\Bhat_i$, $\sup_{\Bhat_i} |\nabla ^m Ric| \leq K_i^{m/2}P_i$, $m=0,1,2$, where $P_i=(n-1) K_i \exp(-A K_i)$, and $A \geq 0$.

(d$'$) If $\Bhat_j \cap \Bhat_i \neq \emptyset$, then $K_j \leq K_i + \Gamma \min \{ \frac{1}{A}, 1 \}$.
\end{assumption}

\begin{rem}
If $A=0$, we can always make (c$'$) ture by choosing $K_i$ large enough.
\end{rem}

\begin{prop}\label{theorem: lifespan estimate - compact case - slow curvature growth}
Let $(M,g)$ be an $n$-dimensional compact Riemannian manifold satisfying Assumption \ref{initial good cover assumption -  slow curvature growth}. Then there exists a constant $c(n,\Gamma)$ depending on $n$ and $\Gamma$, such that the lifespan of the Ricci flow solution on $M$ with initial metric $g$ is greater than
\[
T_0 = \frac{c(n,\Gamma)}{I} \max \left\{ A, \frac{1}{\max_{1\leq i \leq N} \{K_i\}} \right\}.
\]
Moreover,
\[\sup_{\Bhat_i} |Rm|_{g(t)} \leq 2(1+\Gamma)K_i \hspace{4pt} for \hspace{4pt} any \hspace{4pt} t\in [0,T_0] \hspace{4pt} and \hspace{4pt}i=1,2,..., N.\]
\end{prop}
\begin{proof}
In the proof of Proposition \ref{theorem: lifespan estimate - local ricci flow}, the term $e^{A\Gamma}$ is only introduced when we compare $P_i$ and $P_j$; and the term $e^{A\bar{K}}$ is only introduced when directly using $P_i=(n-1)K_ie^{A\bar{K}-AK_i}$. These can be clearly avoided under the new assumptions. (See also Remark \ref{remark: handling A by e^{-A}}.)
\end{proof}

\section{Existence of Ricci flows on non-compact manifolds}

\subsection{Derivative estimates for $Rm$ and $Ric$ under the locall Ricci flow}\label{section: derivatives estimates for the local ricci flow}
Let $(M,g(t))$ be a solution of the local Ricci flow
\begin{equation}\label{equation: local Ricci flow}
\frac{\partial}{\partial t} g = -2 \chi^2 Ric,
\end{equation}
where $\chi$ is a non-negative smooth function supported on a compact domain $\Omega$. In this section we prove some useful regularity estimates for the local Ricci flow. For convenience, we first fix the metrics on a large scale so the curvatures are bounded, then we can easily clarify the dependence on the curvature by scaling arguments.
In this and the following subsection, we will make use of the following evolution equations and inequalities under the local Ricci flow, which have been computed in \cite{xu2013short}. For any integer $m \geq 0$,
\begin{equation} \label{evolution equation of the m^th covariant derivative of Rm}
\begin{split}
\frac{\partial}{\partial t} \nabla^m Rm  =& \chi^2 \Delta \nabla^m Rm + \sum_{i+j+k=m+2, \hspace{4pt} k\leq m+1} \nabla^i \chi \ast \nabla^j \chi \ast \nabla^k Rm  \\
 &  + \sum_{i+j+k=m} \nabla^i(\chi^2) \ast \nabla^j Rm \ast \nabla^k Rm.
 \end{split}
\end{equation}
 And for any integer $m\geq 1$,
 \begin{equation}\label{equation: inequality satisfied by the derivatives of chi}
 \frac{\partial}{\partial t} |\nabla^m \chi| \leq C(n,m) \sum_{k\leq m-1} |\nabla^k (\chi^2 Ric)| |\nabla^{m-k} \chi|.
 \end{equation}

In the following lemmas we assume bounds on the covariant derivatives of $\chi$ in a space-time region. Later, when we apply them to prove existence of the Ricci flow, we will construct cut-off functions satisfying these bounds.

\begin{lem} \label{Shi's estimate - mean value inequality form}
Let $m \geq 1$, and let $0< R, T\leq \alpha$. Suppose there is a constant $C_0$ such that
\[\chi |Rm|(x,t) \leq C_0, \quad  \chi^i |\nabla^i Rm|(x,t) \leq C_0, \quad 1 \leq i \leq m-1;\]
\[  |\nabla^j \chi|_{g(t)}(x,t) \leq C_0, \quad 1\leq j \leq m+1;\]
for any  $(x,t)\in B_{g(0)}(q,R) \times [0,T] $. And assume that the following Sobolev inequality
\[ \left( \int f^\frac{2n}{n-2} d\mu_{g(0)}\right)^\frac{n}{n-2} \leq C_S \int (|\nabla f|_{g(0)}^2 + f^2 )d\mu_{g(0)}\]
holds for any $W^{1,2}$ function $f\geq 0$ compactly supported on $B_{g(0)}(q,R)$.
Then
\[\sup_{B_{g(0)} (q,\frac{R}{2}) \times [\frac{T}{2},T]} \chi^m |\nabla^m Rm| \leq C C_S^\frac{n}{4} (\frac{1}{R^2}+ \frac{1}{T})^\frac{n+2}{4} \| \max\{ \chi^m |\nabla^m Rm|, 1 \} \|_{L^2(\Omega_T)}
\]
for some constant $C$ depending on $\alpha, C_0, n$ and $m$, where $\Omega_T$ denotes the cylinder $B_{g(0)} (q,R) \times [0,T]$.
\end{lem}
\begin{proof}
For any $0 < r \leq \frac{R}{2}$ and $0 < s, \tau \leq \frac{T}{2}$, choose $C^1$ functions $0 \leq \psi, \eta \leq 1$, such that
\begin{equation}\label{spatial cut-off function psi}
\psi(x) =\begin{cases}
1, \quad x\in B_{g(0)}(q,r); \\
0, \quad x\notin B_{g(0)}(q,r+\delta);
\end{cases}
\end{equation}
\[
\eta(t) = \begin{cases}
0, \quad 0 \leq t < s; \\
1, \quad t \geq s + \tau;
\end{cases}
\]
and
\[ |\nabla \psi|^2_{g(0)} \leq \frac{4}{\delta^2} \psi, \quad |\eta'| \leq \frac{2}{\tau}\]
on their supports.
Let $\phi(x,t)=\eta(t) \psi(x)$, then we have $0 \leq \phi \leq 1$ and
\[ |\nabla \phi|_{g(0)}^2 \leq \frac{4}{\delta^2} \phi, \quad |\frac{\partial \phi}{\partial t}| \leq \frac{2}{\tau}\]
on the support of $\phi$.

Since $\chi |Rm|(x,t) \leq C_0$ for all $(x,t)\in B_{g(0)} (q,R) \times [0,T]$, and $T \leq \alpha$, the Riemannian metric remains uniformly equivalent on $B_{g(0)} (q,R)$ for all $0\leq t \leq T$:
\[e^{-2n\alpha C_0} g(0) \leq g(t) \leq e^{2n\alpha C_0} g(0).\]
Hence
\[|\nabla \phi|_{g(t)}(x,t)^2 \leq \frac{C}{\delta^2} \phi(x,t), \]
for any $(x,t)\in B_{g(0)} (q,R) \times [0,T]$. Here and in the following we use $C$ to denote a constant whose value may change from line to line but will only depend on $\alpha, C_0, n$ and $m$.

For any $p \geq 1$, multiply both sides of equation (\ref{evolution equation of the m^th covariant derivative of Rm}) by $ \phi^2 \chi^{2pm} |\nabla^m Rm|^{2p-2} \nabla^m Rm$, and integrate on $B_{g(0)}(q,R)$ to get
\begin{eqnarray*}
&& \int \phi^2 \chi^{2pm} |\nabla^m Rm|^{2p-2} \langle \nabla^m Rm, \frac{\partial}{\partial t} \nabla^m Rm \rangle \\
&=& \int \phi^2 \chi^{2pm+2} |\nabla^m Rm|^{2p-2} \langle  \Delta \nabla^m Rm, \nabla^m Rm \rangle \\
&  & + \int \phi^2 \chi^{2pm+1}|\nabla^m Rm|^{2p-2}  \nabla^m Rm \ast \nabla^{m+2} \chi \ast Rm \\
&   & + \sum_{i+j+k=m+2; \hspace{4pt} i,j,k \leq m+1} \int \phi^2 \chi^{2pm}|\nabla^m Rm|^{2p-2}  \nabla^m Rm \ast \nabla^{i} \chi \ast \nabla^j \chi \ast \nabla^k Rm \\
&& + \sum_{i+j+k=m} \int \phi^2 \chi^{2pm} |\nabla^m Rm|^{2p-2} \nabla^m Rm \ast \nabla^i(\chi^2) \ast \nabla^j Rm \ast \nabla^k  Rm \\
&=& I_1 + I_2 + I_3 + I_4.
\end{eqnarray*}
Apply integration by parts on $I_1$ will yield the good terms containing the $(m+1)$-th  derivative of $Rm$.
\begin{eqnarray*}
I_1 & = & \int \phi^2 \chi^{2mp+2} |\nabla^m Rm|^{2p-2} \langle  \Delta \nabla^m Rm, \nabla^m Rm \rangle \\
     & \leq & -\int \phi^2 \chi^{2mp+2} |\nabla ^m Rm|^{2p-2} |\nabla^{m+1} Rm|^2 \\
     && -\frac{p-1}{2}\int \phi^2 \chi^{2mp+2}|\nabla^m Rm|^{2p-4}|\nabla |\nabla^m Rm|^2|^2\\
     && + 2\int \phi \chi^{2mp+2} |\nabla \phi| |\nabla^m Rm|^{2p-1} |\nabla^{m+1} Rm|. \\
     && + (2pm+2)\int \phi^2 \chi^{2mp+1} |\nabla \chi| |\nabla^m Rm|^{2p-1} |\nabla^{m+1} Rm| \\
     & \leq & -\frac{3}{4}\int \phi^2 \chi^{2mp} |\nabla ^m Rm|^{2p-2} |\nabla^{m+1} Rm|^2 \\
     && -\frac{p-1}{2}\int \phi^2 \chi^{2mp}|\nabla^m Rm|^{2p-4}|\nabla |\nabla^m Rm|^2|^2\\
     && + 8\int (\chi^{2mp+4} |\nabla \phi|^2 + (mp+1)^2 \phi^2 \chi^{2mp+2} |\nabla \chi|^2) |\nabla^m Rm|^{2p}.  \\
\end{eqnarray*}
Similarly we can integrate by parts to get rid of the $(m+2)$-th derivative of $\chi$ in $I_2$ and $I_3$. The $(m+1)$-th derivative of $Rm$ produced in this process can be absorbed by the good term in $I_1$.
\begin{eqnarray*}
I_2 & = & \int \phi^2 \chi^{2mp+1}|\nabla^m Rm|^{2p-2}  \nabla^m Rm \ast \nabla^{m+2} \chi \ast Rm  \\
      & \leq & C(n)\int \phi^2 \chi^{2mp+1}|\nabla^m Rm|^{2p-1}  |\nabla^{m+1} \chi|  |\nabla Rm| \\
      & & + C(n)\int \phi^2 \chi^{2mp+1}|\nabla^m Rm|^{2p-2}  |\nabla^{m+1} Rm|  |\nabla^{m+1} \chi|  | Rm| \\
      && + C(n)(p-1) \int \phi^2 \chi^{2mp+1}|\nabla^m Rm|^{2p-3}  |\nabla |\nabla^m Rm|^2|  |\nabla^{m+1} \chi|  |Rm| \\
      &&  + C(n)\int \phi \chi^{2mp+1} |\nabla \phi| |\nabla^m Rm|^{2p-1}  |\nabla^{m+1} \chi|  | Rm| \\
      && + C(n) (2mp+1) \int \phi^2 \chi^{2mp} |\nabla \chi| |\nabla^m Rm|^{2p-1}  |\nabla^{m+1} \chi|  | Rm| \\
%\end{eqnarray*}
%\begin{eqnarray*}
      & \leq & \frac{1}{4} \int \phi^2 \chi^{2mp} |\nabla^m Rm|^{2p-2} |\nabla^{m+1} Rm|^2 \\
      && + \frac{1}{4}(p-1) \int \phi^2 \chi^{2mp} |\nabla^m Rm|^{2p-4} |\nabla |\nabla^m Rm|^2|^2 \\
     && + C(n)\int \phi^2 \chi^{2mp+1}|\nabla^m Rm|^{2p-1}  |\nabla^{m+1} \chi|  |\nabla Rm| \\
     && + C(n) \int (\phi |\nabla \phi|\chi + (2mp+1)\phi^2 |\nabla \chi| ) \chi^{2mp} |\nabla^m Rm|^{2p-1}  |\nabla^{m+1} \chi|  | Rm| \\
      && + C(n)p \int \phi^2 \chi^{2mp+2}   |\nabla^{m+1} \chi|^2  |Rm|^2 |\nabla^m Rm|^{2p-2}  . \\
\end{eqnarray*}

\begin{eqnarray*}
I_3 & = & \sum_{i+j+k=m+2; \hspace{4pt} i,j,k \leq m+1} \int \phi^2 \chi^{2mp}|\nabla^m Rm|^{2p-2}  \nabla^m Rm \ast \nabla^{i} \chi \ast \nabla^j \chi \ast \nabla^k Rm \\
 & \leq &  C(n) \int \phi^2 \chi^{2mp+1}|\nabla \chi| |\nabla^m Rm|^{2p-1} |\nabla^{m+1} Rm| \\
 && + C(n) \sum_{i+j+k=m+2; \hspace{4pt} i,j,k \leq m} \int \phi^2 \chi^{2mp} |\nabla^i \chi| |\nabla^j \chi| |\nabla^m Rm|^{2p-1} |\nabla^k Rm| \\
 &\leq &  \frac{1}{4} \int \phi^2 \chi^{2mp} |\nabla^m Rm|^{2p-2} |\nabla^{m+1} Rm|^2 \\
 && + C(n) \int \phi^2 \chi^{2mp}(\chi^2 |\nabla \chi|^2 + \chi |\nabla^2 \chi| +  |\nabla \chi|^2)|\nabla^m Rm|^{2p} \\
 && + C(n) \sum_{i+j+k=m+2; \hspace{4pt} i,j,k \leq m-1} \int \phi^2 \chi^{2mp} |\nabla^i \chi| |\nabla^j \chi| |\nabla^m Rm|^{2p-1} |\nabla^k Rm|
\end{eqnarray*}
$I_4$ can be simply estimated by the Cauchy inequality.
\begin{eqnarray*}
I_4 & = & \sum_{i+j+k=m} \int \phi^2 \chi^{2mp} |\nabla^m Rm|^{2p-2} \nabla^m Rm \ast \nabla^i(\chi^2) \ast \nabla^j Rm \ast \nabla^k  Rm \\
& \leq & C(n) \int \phi^2 \chi^{2mp+2} |Rm| |\nabla^m Rm|^{2p} \\
&& + C(n) \sum_{i+j+k=m; \hspace{4pt} j, k\leq m-1} \int \phi^2 \chi^{2mp} |\nabla^i (\chi^2)| |\nabla^j Rm| |\nabla^k Rm| |\nabla^m Rm|^{2p-1}.
\end{eqnarray*}
By the assumptions of the lemma, we can replace the lower order ($\leq m-1$) derivatives of $Rm$, and lower order ($\leq m+1$) derivatives of $\chi$ by their corresponding bounds.
\begin{eqnarray*}
I_1  & \leq & -\frac{3}{4}\int \phi^2 \chi^{2mp} |\nabla ^m Rm|^{2p-2} |\nabla^{m+1} Rm|^2 \\
     && -\frac{p-1}{2}\int \phi^2 \chi^{2mp}|\nabla^m Rm|^{2p-4}|\nabla |\nabla^m Rm|^2|^2\\
     && + C ( p^2 + \frac{1}{\delta^2} ) \int \phi \chi^{2mp} |\nabla^m Rm|^{2p}.  \\
\end{eqnarray*}
\begin{eqnarray*}
 I_2 & \leq & \frac{1}{4} \int \phi^2 \chi^{2mp} |\nabla^m Rm|^{2p-2} |\nabla^{m+1} Rm|^2 \\
      && + \frac{1}{4}(p-1) \int \phi^2 \chi^{2mp} |\nabla^m Rm|^{2p-4} |\nabla |\nabla^m Rm|^2|^2 \\
      && + C ( p + \frac{1}{\delta}) \int \phi\chi^{2mp-1}  |\nabla^m Rm|^{2p-1}  + Cp \int \phi^2 \chi^{2mp} |\nabla^m Rm|^{2p-2}  . \\
\end{eqnarray*}
\begin{eqnarray*}
I_3 &\leq &  \frac{1}{4} \int \phi^2 \chi^{2mp} |\nabla^m Rm|^{2p-2} |\nabla^{m+1} Rm|^2 \\
 && + C \int \phi^2 \chi^{2mp} |\nabla^m Rm|^{2p} + C \int \phi^2 \chi^{2mp-m}  |\nabla^m Rm|^{2p-1}
\end{eqnarray*}
\begin{eqnarray*}
I_4 & \leq & C \int \phi^2 \chi^{2mp} |\nabla^m Rm|^{2p}  + C \int \phi^2 \chi^{2mp-m}  |\nabla^m Rm|^{2p-1}.
\end{eqnarray*}
%When $m =1$ we have
%\begin{eqnarray*}
%I_4 & \leq & C(n) \int \phi^2 \chi^{2mp+2} |Rm| |\nabla Rm|^{2p} \\
%&& + C(n) \int \phi^2 \chi^{2mp} |\nabla (\chi^2)| | Rm|^2 |\nabla Rm|^{2p-1}. \\
%& \leq & C(n) \int \phi^2 \chi^{2mp} |\nabla Rm|^{2p} + C(n) \int \phi^2 \chi^{2mp-m} |\nabla Rm|^{2p-1}.
%\end{eqnarray*}
%When $m=2$,
%\begin{eqnarray*}
%I_4 & \leq & C(n) \int \phi^2 \chi^{2mp+2} |Rm| |\nabla^2 Rm|^{2p} \\
%&& + C(n) \int \phi^2 \chi^{2mp+2} |\nabla Rm|^2 |\nabla^2 Rm|^{2p-1} \\
%&& + C(n) \int \phi^2 \chi^{2mp+1} |\nabla \chi| |Rm| |\nabla Rm| |\nabla^2 Rm|^{2p-1} \\
%&& + C(n) \int \phi^2 \chi^{2mp} |\nabla \chi|^2 |Rm|^2 |\nabla^2 Rm|^{2p-1} \\
%&& + C(n) \int \phi^2 \chi^{2mp+1} |\nabla^2 \chi| |Rm|^2 |\nabla^2 Rm|^{2p-1}\\
%& \leq & C(n) \int \phi^2 \chi^{2mp} |\nabla^2 Rm|^{2p} + C(n) \int \phi^2 \chi^{2mp-m} |\nabla^2 Rm|^{2p-1}
%\end{eqnarray*}
Using the above estimates we can control the time derivative of spatial integrals of $\chi |\nabla^m Rm|$.
\begin{eqnarray*}
\frac{\partial}{\partial t} \int \phi^2 \chi^{2mp} |\nabla^m Rm|^{2p} & = & 2p\int \phi^2 \chi^{2mp} |\nabla^m Rm|^{2p-2} \langle \nabla^m Rm, \frac{\partial}{\partial t} \nabla^m Rm \rangle \\
&& + p\int \phi^2 \chi^{2mp} |\nabla^m Rm|^{2p-2} Ric \ast \nabla^m Rm \ast \nabla^m Rm \\
&& + 2 \int \phi \frac{\partial \phi}{\partial t} \chi^{2mp} |\nabla^m Rm|^{2p} - \int \phi^2 \chi^{2mp} |\nabla^m Rm|^{2p} S \\
\end{eqnarray*}
\begin{eqnarray*}
& \leq & -\frac{p}{2}\int \phi^2 \chi^{2mp} |\nabla ^m Rm|^{2p-2} |\nabla^{m+1} Rm|^2 \\
     && -\frac{p(p-1)}{2}\int \phi^2 \chi^{2mp}|\nabla^m Rm|^{2p-4}|\nabla |\nabla^m Rm|^2|^2\\
     && + C (p^3 + \frac{p}{\delta^2} + \frac{1}{\tau}) \int \phi \chi^{2mp} |\nabla^m Rm|^{2p} \\
     && + C ( p^2 + \frac{p}{\delta} ) \int \phi \chi^{2mp-m}  |\nabla^m Rm|^{2p-1}  \\
     && + Cp^2 \int \phi^2 \chi^{2mp-2m} |\nabla^m Rm|^{2p-2}  . \\
\end{eqnarray*}
Interpolate using the Cauchy inequality
\[\int \phi \chi^{2mp-m}  |\nabla^m Rm|^{2p-1} \leq \int \phi \chi^{2mp} |\nabla^m Rm|^{2p} +\int \phi \chi^{2mp-2m} |\nabla^m Rm|^{2p-2} . \]
Then integrate on $[0,t]$ for each $t\in (0,T]$ to get
\begin{eqnarray*}
&& \sup_{0 \leq t \leq T}\int \phi^2 \chi^{2mp} |\nabla^m Rm|^{2p} (t) + \frac{p}{2} \int_0^T \int \phi^2 \chi^{2mp} |\nabla ^m Rm|^{2p-2} |\nabla^{m+1} Rm|^2 \\
&& + \frac{p(p-1)}{2} \int_0^T \int \phi^2 \chi^{2mp}|\nabla^m Rm|^{2p-4}|\nabla |\nabla^m Rm|^2|^2  \\
& \leq & C (p^3 + \frac{p}{\delta^2} + \frac{1}{\tau}) \int_0^T \int \phi \chi^{2mp} |\nabla^m Rm|^{2p} \\
&&+ C( p^2 + \frac{p}{\delta} ) \int_0^T \int \phi \chi^{2mp-2m}  |\nabla^m Rm|^{2p-2}  \\
& \leq & C (p^3 + \frac{p}{\delta^2} + \frac{1}{\tau}) \int_0^T \int \phi u^{2p},
%& \leq & C (p^3 + \frac{p}{\delta^2} + \frac{1}{\tau}) \int_0^T \int \phi \chi^{2mp} |\nabla^m Rm|^{2p} \\
%&& + C( p^2 + \frac{p}{\delta} )  \left( \int_0^T \int \phi \chi^{2mp} |\nabla^m Rm|^{2p} \right)^\frac{p-1}{p} \left( \int_0^T Vol_{g(t)} spt(\psi) dt \right)^\frac{1}{p} \\
\end{eqnarray*}
where
\[ u=\max\{ \chi^m |\nabla^m Rm|, 1\}. \]

By the equivalence of metrics, the Sobolev inequality holds for any time $t \in [0,T]$ with a possibly larger Sobolev constant $e^{C(n)C_0\alpha}C_S$. Apply the Sobolev inequality to  $\phi u^p$, and note that $|\nabla u| \leq |\nabla (\chi^m |\nabla^m Rm| )|$ in the sense of distribution, we get from the above inequality that
\[
\left( \int (\phi u^p)^\frac{2n}{n-2} \right)^\frac{n-2}{n} \leq C_S Q,
\]
where
\begin{eqnarray*}
Q & = & \bar{C} (p^3 + \frac{p}{\delta^2} + \frac{1}{\tau}) \int_0^T \int \phi u^{2p}, \\
%&& + C( p^2 + \frac{p}{\delta} ) T^\frac{1}{p} \left( \int_0^T \int \phi\chi^{2mp} |\nabla^m Rm|^{2p} \right)^\frac{p-1}{p} .
\end{eqnarray*}
$\bar{C}$ depends only on $n, m, \alpha$ and $C_0$.
Then the Holder inequality implies
\[
\begin{split}
\int_0^T \int (\phi u^p)^\frac{2(n+2)}{n}  \leq \int_0^T \left( \int (\phi u^p)^\frac{2n}{n-2} \right)^\frac{n-2}{n} \left( \int \phi^2 u^{2p} \right)^\frac{2}{n} \leq   C_S Q^\frac{n+2}{n}.
\end{split}
\]
Hence
\[
\left\| u \right\|_{L^\frac{2p(n+2)}{n} (r, s + \tau) } \leq C_S^{\frac{n}{2p(n+2)}}\bar{C}^\frac{1}{2p}p^\frac{3}{2p}\max\{ \frac{1}{\delta^\frac{1}{p}} , \frac{1}{\tau^\frac{1}{2p}}\} \left\| u \right\|_{L^{2p}(r+\delta, s) } ,
\]
where $\|\cdot\|_{L^q(r,s)}$ denotes the $L^q$-norm on $B_{g(0)}(q,r) \times [s, T])$.

Choose $p_i=(\frac{n+2}{n})^i$, $r_0=R$, $\delta_i = \frac{1}{2^{i+2}}R$, $r_{i+1}=r_i -\delta_i$, $s_0 =0$, $\tau_i=\frac{1}{2^{i+2}}T$, $s_{i+1}=s_i+\tau_{i}$, where $i=0,1,2,...$ Iteration of the above inequality yields the lemma.
\end{proof}

If we assume an initial bound of $|\nabla^m Rm|_{g(0)}$, we can get the following modified version of Lemma \ref{Shi's estimate - mean value inequality form}.
\begin{lem} \label{lemma: Modified Shi's estimate - mean value inequality form}
Suppose in addition to the assumptions of Lemma \ref{Shi's estimate - mean value inequality form}, we also have
\[ \sup_{B_{g(0)}(q,R)} \chi^m |\nabla^m Rm|_{g(0)} \leq C_1.\]
Then
\[\sup_{B_{g(0)} (q,\frac{R}{2}) \times [0,T]} \chi^m |\nabla^m Rm| \leq  \frac{CC_S^\frac{n}{4}}{R^\frac{n+2}{2}}\max \left\{ \| \max \{\chi^m |\nabla^m Rm|, 1\} \|_{L^2(\Omega_T)}, \hspace{2pt} \sqrt{V}C_1 \right \},
\]
where $V = Vol_{g(0)}(B_{g(0)}(q, R))$, $\Omega_T$ denotes the cylinder $ B_{g(0)} (q,R) \times [0,T]$, and $C$ is a constant depending only on $\alpha, C_0, C_S, n, m$,
\end{lem}
\begin{proof}
We only need to modify the proof of Lemma \ref{Shi's estimate - mean value inequality form}.
Use $\psi$ in (\ref{spatial cut-off function psi}) to be the cut-off function. Note that $\psi$ is independent of $t$. We can apply the same arguments as in the proof of Lemma \ref{Shi's estimate - mean value inequality form} to arrive at
\[
\begin{split}
\frac{\partial}{\partial t} \int \psi^2 \chi^{2mp} |\nabla^m Rm|^{2p}  = & 2p\int \psi^2 \chi^{2mp} |\nabla^m Rm|^{2p-2} \langle \nabla^m Rm, \frac{\partial}{\partial t} \nabla^m Rm \rangle \\
& + p\int \psi^2 \chi^{2mp} |\nabla^m Rm|^{2p-2} Ric \ast \nabla^m Rm \ast \nabla^m Rm \\
 \leq & -\frac{p}{2}\int \psi^2 \chi^{2mp} |\nabla ^m Rm|^{2p-2} |\nabla^{m+1} Rm|^2 \\
     & -\frac{p(p-1)}{2}\int \psi^2 \chi^{2mp}|\nabla^m Rm|^{2p-4}|\nabla |\nabla^m Rm|^2|^2\\
     & + C (p^3 + \frac{p}{\delta^2}) \int \psi \chi^{2mp} |\nabla^m Rm|^{2p}  \\
     & + C ( p^2 + \frac{p}{\delta} ) \int \psi \chi^{2mp-2m}  |\nabla^m Rm|^{2p-2}.  \\
\end{split}
\]
Integrate on $[0,T]$ to get
\begin{eqnarray*}
&& \sup_{0 \leq t \leq T}\int \psi^2 \chi^{2mp} |\nabla^m Rm|^{2p} (t) + \frac{p}{2} \int_0^T \int \psi^2 \chi^{2mp} |\nabla ^m Rm|^{2p-2} |\nabla^{m+1} Rm|^2 \\
&& + \frac{p(p-1)}{2} \int_0^T \int \psi^2 \chi^{2mp}|\nabla^m Rm|^{2p-4}|\nabla |\nabla^m Rm|^2|^2  \\
& \leq & \int \psi^2 \chi^{2mp} |\nabla^m Rm|^{2p} (0) + C (p^3 + \frac{p}{\delta^2}) \int_0^T \int \psi u^{2p},
\end{eqnarray*}
where $u=\max\{\chi |\nabla^m Rm|, 1\}$.

Apply Holder and Sobolev inequalities as in the previous proof, we get
\[
\int_0^T \int (\phi u^p)^\frac{2(n+2)}{n} \leq C_S Q^\frac{n+2}{n},
\]
where
\[
Q  =  C (p^3 + \frac{p}{\delta^2} ) \int_0^T \int \psi u^{2p} + C \int \psi^2 \chi^{2mp} |\nabla^m Rm|^{2p} (0).
\]
Let
\[
F(r, p) : = \max \left\{ \left\|u \right\|_{L^{2p} (B_{g(0)}(q,r) \times [0, T]) } , \hspace{4pt} V(r)^\frac{1}{2p} C_1 \right\},
\]
where $V(r)=Vol_{g(0)}(B_{g(0)}(q,r))$.
We have
\[
\begin{split}
Q  \leq & C (p^3 + \frac{p}{\delta^2} ) F(r+\delta,p)^{2p} + C F(r+\delta, p)^{2p}\\
 \leq & C p^3\frac{1}{\delta^2} F(r+\delta, p)^{2p}.
\end{split}
\]
Hence
\[
\left\|u \right\|_{L^\frac{2p(n+2)}{n} (B_{g(0)}(q,r) \times [0, T]) } \leq C_S^\frac{n}{2p(n+2)}C^\frac{1}{2p}p^\frac{3}{2p} \delta^{-\frac{1}{p}}  F(r+\delta,p).
\]
Apply the Sobolev inequality with the test function $\psi$, we get
\[
V(r)^\frac{n-2}{n}\leq C_S(\frac{1}{\delta^2} (V(r+\delta) - V(r)) + V(r+\delta)).
\]
WLOG we can assume that $\delta < R < 1$, otherwise we can cover $B_{g(0)}(q,R)$ by geodesic balls with radius $1$ and work on each smaller ball instead. Then
\[
V(r)^\frac{n-2}{n}\leq \frac{2C_S}{\delta^2} V(r+\delta),
\]
and
\[V(r)= V(r)^\frac{n-2}{n}V(r)^\frac{2}{n}\leq \frac{2C_S}{\delta^2}V(r+\delta)^\frac{n+2}{n},\]
hence
\[V(r)^\frac{n}{2p(n+2)} C_1 \leq (2C_S)^\frac{n}{2p(n+2)} \delta^{-\frac{1}{p}}F(r+\delta, p).\]
Therefore
\[
F(r, \frac{(n+2)p}{n}) \leq C_S^\frac{n}{2p(n+2)} \bar{C}^\frac{1}{2p}p^\frac{3}{2p} \delta^{-\frac{1}{p}} F(r+\delta, p),
\]
for some constant $\bar{C}$ depending on $\alpha, C_0, n$ and $m$.
Then we can iterate this inequality to finish the proof.
\end{proof}

Now we show how to obtain $L^2$ control of $\chi^m |\nabla^m Rm|$ .

\begin{lem}\label{lemma: L2 control of the derivatives of curvature}
Let $m\geq 1$, $0 < R, T \leq \alpha$. Suppose there is a constant $C_0$ such that
\[  \chi^{i} |\nabla^{i} Rm| (x,t) \leq C_0, \quad i=0,1,...,m-1;\]
\[ |\nabla^j \chi|_{g(t)}(x,t) \leq C_0, \quad j=1,2,...,m+1;\]
for all  $(x,t)\in B_{g(0)}(q,2R) \times [0,T] $. Then
\[
\|\chi^m |\nabla^m Rm| \|_{L^2(B_{g(0)}(q,R) \times [0,T])} \leq C\sqrt{Vol_{g(0)}(B_{g(0)}(q,2R))},
\]
where $C$ depends on $\alpha, C_0, n$ and $m$.
\end{lem}
\begin{proof}
Choose a $C^1$ function $0 \leq \psi \leq 1$, such that
\begin{equation*}
\psi(x) =\begin{cases}
1, \quad x\in B_{g(0)}(q,R); \\
0, \quad x\notin B_{g(0)}(q, 2R);
\end{cases}
\end{equation*}
and
\[ |\nabla \psi|^2_{g(0)} \leq \frac{4}{R^2} \psi\]
on its support. Using the evolution equation of $\nabla^{m-1} Rm$, we can calculate
\[
\begin{split}
&\frac{\partial}{\partial t} \int \psi^2 \chi^{2(m-1)} |\nabla^{m-1} Rm|^2  \\
=& \int \psi^2 \chi^{2(m-1)}\langle \nabla^{m-1} Rm, \chi^2 \Delta \nabla^{m-1} Rm\rangle \\
& +\sum_{i+j+k=m+1, \hspace{4pt} k\leq m}\int \psi^2 \chi^{2(m-1)}\langle \nabla^{m-1} Rm,  \nabla^i \chi \ast \nabla^j \chi \ast \nabla^k Rm  \rangle \\
 & + \sum_{i+j+k=m-1} \int \psi^2 \chi^{2(m-1)}\langle \nabla^{m-1} Rm,  \nabla^i(\chi^2) \ast \nabla^j Rm \ast \nabla^k Rm\rangle \\
\leq & - \int \psi^2 \chi^{2m} |\nabla^m Rm|^2 + \int |\nabla (\psi^2 \chi^{2m})| |\nabla^{m-1} Rm| |\nabla^m Rm| \\
        & + \int \psi^2 \chi^{2m-1} |\nabla \chi| |\nabla^{m-1} Rm| |\nabla^m Rm|  \\
        & + C(n) \sum_{i+j+k=m+1, \hspace{4pt} k\leq m-1}\int \psi^2 \chi^{2(m-1)}|\nabla^i \chi| | \nabla^j \chi|  |\nabla^k Rm | |\nabla^{m-1} Rm| \\
        & + C(n) \sum_{i+j+k=m-1} \int \psi^2 \chi^{2(m-1)}  |\nabla^i(\chi^2)| |\nabla^j Rm|  |\nabla^k Rm| |\nabla^{m-1} Rm| \\
\end{split}
\]
By the Cauchy inequality
\[
\frac{\partial}{\partial t} \int \psi^2 \chi^{2(m-1)} |\nabla^{m-1} Rm|^2  \leq  - \frac{1}{4}\int \psi^2 \chi^{2m} |\nabla^m Rm|^2 +I,
\]
where
\[
\begin{split}
I =&  \frac{8}{R^2}\int \psi \chi^2 \chi^{2(m-1)} |\nabla^{m-1} Rm|^2 + 2\int \psi^2 | \nabla \chi|^2 \chi^{2(m-1)} |\nabla^{m-1} Rm|^2 \\
        & + \int \psi^2 \chi^{2m-2} |\nabla \chi|^2 |\nabla^{m-1} Rm|^2 \\
        & + C(n) \sum_{i+j+k=m+1, \hspace{4pt} k\leq m-1}\int \psi^2 \chi^{2(m-1)}|\nabla^i \chi| | \nabla^j \chi|  |\nabla^k Rm | |\nabla^{m-1} Rm| \\
        & + C(n) \sum_{i+j+k=m-1} \int \psi^2 \chi^{2(m-1)}  |\nabla^i(\chi^2)| |\nabla^j Rm|  |\nabla^k Rm| |\nabla^{m-1} Rm|. \\
\end{split}
\]
Integrate on $[0,T]$, we get
\[
\frac{1}{4} \int_0^T \int \psi^2 \chi^{2m} |\nabla^m Rm|^2 \leq  \int \psi^2 \chi^{2(m-1)} |\nabla^{m-1} Rm|^2 (0) + \int_0^T I dt.
\]
Then the result easily follows from the assumptions of the lemma.
\end{proof}

By lemma \ref{Shi's estimate - mean value inequality form}, \ref{lemma: Modified Shi's estimate - mean value inequality form}, \ref{lemma: L2 control of the derivatives of curvature}, and scaling arguments, we can prove the following:
\begin{lem} \label{lemma: Shi's estimate for the local ricci flow}
Let $m \geq 1$, $K >0$, and let $0< R \leq \frac{\alpha}{\sqrt{K}}, 0< T\leq \frac{\alpha}{K}$. Suppose there is a constant $C_0$ such that
\[  \chi^i |\nabla^i Rm|(x,t) \leq C_0 K^{1+\frac{i}{2}}, \quad 0 \leq i \leq m-1;\]
\[  |\nabla^j \chi|_{g(t)}(x,t) \leq C_0 K^\frac{j}{2}, \quad 1\leq j \leq m+1;\]
for any  $(x,t)\in B_{g(0)}(q,2R) \times [0,T] $.
Then
\[\sup_{B_{g(0)} (q,\frac{R}{2}) \times [\frac{T}{2},T]} \chi^m |\nabla^m Rm| \leq C (\frac{1}{R^2}+ \frac{1}{T})^\frac{n+2}{4} K^{\frac{2m+2 -n}{4}},
\]
for some constant $C$ depending on $\alpha, C_0, n$ and $m$.
\end{lem}
\begin{proof}
Let $\tilde{g}=Kg$. Then $Ric(\tilde{g}(0)) \geq - C_0 \tilde{g}(0)$. The Sobolev inequality comes from Lemma \ref{lemma: Sobolev inequality}. The rest of the proof is a direct application of Lemma \ref{Shi's estimate - mean value inequality form} and Lemma \ref{lemma: L2 control of the derivatives of curvature}.
\end{proof}
\begin{lem} \label{lemma: Modified Shi's estimate for the local Ricci flow}
Under the same assumptions as in Lemma \ref{lemma: Shi's estimate for the local ricci flow}, and assume in addition that
\[ \sup_{B_{g(0)}(q,2R)} \chi^m |\nabla^m Rm|_{g(0)} \leq C_1 K^\frac{m+2}{2},\]
where $C_1 \geq 1$.
Then
\[\sup_{B_{g(0)} (q,\frac{R}{2}) \times [0,T]} \chi^m |\nabla^m Rm| \leq C R^{-\frac{n+2}{2}} K^{\frac{2m+2 -n}{4}},
\]
where $C$ depends on $\alpha, C_0, C_1, n$ and $m$.
\end{lem}
\begin{proof}
Similar to the proof of Lemma \ref{lemma: Shi's estimate for the local ricci flow}.
\end{proof}

In the remainder of this subsection, we will assume that $Rm$ is bounded up to the $m$-th derivative, then we will derive derivative estimates of $Ric$. The point is, $Ric$ and its derivatives may be much smaller compared to $Rm$. Similar to (\ref{evolution equation of the m^th covariant derivative of Rm}), we can compute that the Ricci curvature and its covariant derivatives evolve by
\begin{equation} \label{evolution equation of the m^th covariant derivative of Ric}
\begin{split}
\frac{\partial}{\partial t} \nabla^m Ric  =& \chi^2 \Delta \nabla^m Ric + \sum_{i+j+k=m+2, \hspace{4pt} k\leq m+1} \nabla^i \chi \ast \nabla^j \chi \ast \nabla^k Ric  \\
 &  + \sum_{i+j+k=m} \nabla^i(\chi^2) \ast \nabla^j Rm \ast \nabla^k Ric,
 \end{split}
\end{equation}
where the integer $m\geq 0$.
\begin{lem}\label{mean value inequality for derivatives of Ric}
Let $m \geq 1$, and let $R >0, T >0$. Suppose there are constants $C_0, C_1, C_2$ such that
\[\chi |Rm|(x,t) \leq C_0, \quad  \chi^i |\nabla^i Rm|(x,t) \leq C_0, \quad 1 \leq i \leq m;\]
\[  |\nabla^j \chi|_{g(t)}(x,t) \leq C_0, \quad 1\leq j \leq m+1;\]
\[\chi |Ric|(x,t) \leq C_1, \quad \chi^k |\nabla^k Ric|(x,t) \leq C_1, \quad 1 \leq k \leq m-1;\]
and
\[\chi^m |\nabla^m Ric|(x,0) \leq C_1;\]
for any  $x\in B_{g(0)}(q,R)$ and $t \in [0,T] $. Assume that the following Sobolev inequality
\[ \left( \int f^\frac{2n}{n-2} d\mu_{g(t)}\right)^\frac{n}{n-2} \leq C_S \int (|\nabla f|_{g(t)}^2 + f^2 )d\mu_{g(t)}\]
holds for any $W^{1,2}$ function $f\geq 0$ compactly supported on $B_{g(0)}(q,R)$. And assume that
\[Vol_{g(t)} B_{g(0)}(q,r) \leq C_2 V_r , \quad t\in[0,T], \quad 0<r\leq R,\]
where $V_r= Vol_{g(0)}(B_{g(0)}(q,r))$.
Then
\[\sup_{B_{g(0)} (q,\frac{R}{2}) \times [0,T]} \chi^m |\nabla^m Ric| \leq  C_S^\frac{n}{4}\frac{C\max\{T,1\}}{R^2}\max \left\{ \| \chi^m |\nabla^m Ric| \|_{L^2(\Omega_T)}, \hspace{2pt} \sqrt{V_R}C_1 \right \}
\]
for some constant $C$ depending on $ C_0, C_2, n$ and $m$, where $\Omega_T=B_{g(0)} (q,R) \times [0,T]$.
\end{lem}
\begin{proof}
Using the same $\psi$ as in (\ref{spatial cut-off function psi}), we can calculate similarly as in the proof of Lemma \ref{lemma: Modified Shi's estimate - mean value inequality form} that
\begin{eqnarray*}
&& \int \psi^2 \chi^{2pm} |\nabla^m Ric|^{2p-2} \langle \nabla^m Ric, \frac{\partial}{\partial t} \nabla^m Ric \rangle \\
&=& \int \psi^2 \chi^{2pm+2} |\nabla^m Ric|^{2p-2} \langle  \Delta \nabla^m Ric, \nabla^m Ric \rangle \\
&  & + \int \psi^2 \chi^{2pm+1}|\nabla^m Ric|^{2p-2}  \nabla^m Ric \ast \nabla^{m+2} \chi \ast Ric \\
&   & + \sum_{i+j+k=m+2; \hspace{4pt} i,j,k \leq m+1} \int \psi^2 \chi^{2pm}|\nabla^m Ric|^{2p-2}  \nabla^m Ric \ast \nabla^{i} \chi \ast \nabla^j \chi \ast \nabla^k Ric \\
&& + \sum_{i+j+k=m} \int \psi^2 \chi^{2pm} |\nabla^m Ric|^{2p-2} \nabla^m Ric \ast \nabla^i(\chi^2) \ast \nabla^j Rm \ast \nabla^k  Ric \\
&=& I_1 + I_2 + I_3 + I_4.
\end{eqnarray*}
$I_1,I_2,I_3, I_4$ can be handled similarly as in the proof of Lemma \ref{Shi's estimate - mean value inequality form}, except that we have to keep track of $C_1$ whenever we use the bound $\chi^k|\nabla^k Ric|\leq C_1$ for $1\leq k\leq m-1$.
\begin{eqnarray*}
I_1 & = & \int \psi^2 \chi^{2mp+2} |\nabla^m Ric|^{2p-2} \langle  \Delta \nabla^m Ric, \nabla^m Ric \rangle \\
     & \leq & -\frac{3}{4}\int \psi^2 \chi^{2mp} |\nabla ^m Ric|^{2p-2} |\nabla^{m+1} Ric|^2 \\
     && -\frac{p-1}{2}\int \psi^2 \chi^{2mp}|\nabla^m Ric|^{2p-4}|\nabla |\nabla^m Ric|^2|^2\\
     && + 8\int (\chi^{2mp+4} |\nabla \psi|^2 + (mp+1)^2 \psi^2 \chi^{2mp+2} |\nabla \chi|^2) |\nabla^m Ric|^{2p}  \\
     & \leq & -\frac{3}{4}\int \psi^2 \chi^{2mp} |\nabla ^m Ric|^{2p-2} |\nabla^{m+1} Ric|^2 \\
     && -\frac{p-1}{2}\int \psi^2 \chi^{2mp}|\nabla^m Ric|^{2p-4}|\nabla |\nabla^m Ric|^2|^2\\
     && + C(n,m,C_0)(p^2+\frac{1}{\delta^2})\int \psi \chi^{2mp}|\nabla^m Ric|^{2p}  \\
\end{eqnarray*}
\begin{eqnarray*}
I_2 & = & \int \psi^2 \chi^{2mp+1}|\nabla^m Ric|^{2p-2}  \nabla^m Ric \ast \nabla^{m+2} \chi \ast Ric  \\
      & \leq & \frac{1}{4} \int \psi^2 \chi^{2mp} |\nabla^m Ric|^{2p-2} |\nabla^{m+1} Ric|^2 \\
      && + \frac{1}{4}(p-1) \int \psi^2 \chi^{2mp} |\nabla^m Ric|^{2p-4} |\nabla |\nabla^m Ric|^2|^2 \\
     && + C(n)\int \psi^2 \chi^{2mp+1}|\nabla^m Ric|^{2p-1}  |\nabla^{m+1} \chi|  |\nabla Ric| \\
     && + C(n) \int (\psi |\nabla \psi|\chi + (2mp+1)\psi^2 |\nabla \chi| ) \chi^{2mp} |\nabla^m Ric|^{2p-1}  |\nabla^{m+1} \chi|  | Ric| \\
      && + C(n)p \int \psi^2 \chi^{2mp+2}   |\nabla^{m+1} \chi|^2  |Ric|^2 |\nabla^m Ric|^{2p-2}  \\
 & \leq & \frac{1}{4} \int \psi^2 \chi^{2mp} |\nabla^m Ric|^{2p-2} |\nabla^{m+1} Ric|^2 \\
      && + \frac{1}{4}(p-1) \int \psi^2 \chi^{2mp} |\nabla^m Ric|^{2p-4} |\nabla |\nabla^m Ric|^2|^2 \\
     && + C(n,C_0)C_1\int \psi^2 \chi^{2mp-m}|\nabla^m Ric|^{2p-1} \\
     && + C(n,m,C_0) C_1\int \psi \chi^{2mp-m} |\nabla^m Ric|^{2p-1} \\
      && + C(n,C_0)p C_1^2 \int \psi^2 \chi^{2mp-2m}  |\nabla^m Ric|^{2p-2}  \\
\end{eqnarray*}

\begin{eqnarray*}
I_3 & = & \sum_{i+j+k=m+2; \hspace{4pt} i,j,k \leq m+1} \int \psi^2 \chi^{2mp}|\nabla^m Ric|^{2p-2}  \nabla^m Ric \ast \nabla^{i} \chi \ast \nabla^j \chi \ast \nabla^k Ric \\
 &\leq &  \frac{1}{4} \int \psi^2 \chi^{2mp} |\nabla^m Ric|^{2p-2} |\nabla^{m+1} Ric|^2 \\
 && + C(n,C_0) \int \psi^2 \chi^{2mp}|\nabla^m Ric|^{2p} \\
 && + C(n,m,C_0) C_1  \int \psi^2 \chi^{2mp-m} |\nabla^m Ric|^{2p-1}
\end{eqnarray*}
\begin{eqnarray*}
I_4 & = & \sum_{i+j+k=m} \int \psi^2 \chi^{2mp} |\nabla^m Ric|^{2p-2} \nabla^m Ric \ast \nabla^i(\chi^2) \ast \nabla^j Rm \ast \nabla^k  Ric \\
& \leq & C(n) \int \psi^2 \chi^{2mp+2} |Rm| |\nabla^m Ric|^{2p} \\
&& + C(n) \sum_{i+j+k=m; \hspace{4pt} k\leq m-1} \int \psi^2 \chi^{2mp} |\nabla^i (\chi^2)| |\nabla^j Rm| |\nabla^k Ric| |\nabla^m Ric|^{2p-1} \\
& \leq & C(n,C_0) \int \psi^2 \chi^{2mp} |\nabla^m Ric|^{2p}  + C(n,m,C_0) C_1 \int \psi^2 \chi^{2mp-m}  |\nabla^m Ric|^{2p-1} \\
& \leq & C(n,m,C_0) \int \psi^2 \chi^{2mp} |\nabla^m Ric|^{2p}  + C(n,m,C_0) C_1^2 \int \psi^2 \chi^{2mp-2m}  |\nabla^m Ric|^{2p-2}
\end{eqnarray*}
In the second last inequality we have used the assumptions of the lemma, in particular $\chi^m |\nabla^m Rm| \leq C_0$.

Similar as in the previous proofs we can integrate to get
\[
\begin{split}
& \sup_{0 \leq t \leq T}\int \psi^2 \chi^{2mp} |\nabla^m Ric|^{2p} (t) + \frac{p}{2} \int_0^T \int \psi^2 \chi^{2mp} |\nabla ^m Ric|^{2p-2} |\nabla^{m+1} Ric|^2 \\
& + \frac{p(p-1)}{2} \int_0^T \int \psi^2 \chi^{2mp}|\nabla^m Ric|^{2p-4}|\nabla |\nabla^m Ric|^2|^2  \\
 \leq & \int \psi^2 \chi^{2mp} |\nabla^m Ric|^{2p} (0) + C (p^3 + \frac{p}{\delta^2}) \int_0^T \int \psi \chi^{2mp} |\nabla^m Ric|^{2p} \\
& + C( p^2 + \frac{1}{\delta} ) C_1^2 \left( \int_0^T \int \psi \chi^{2mp} |\nabla^m Ric|^{2p} \right)^\frac{p-1}{p} (V_{r+\delta}T)^\frac{1}{p} \\
 \leq & C V_{r+\delta}T C_1^{2p} + C \frac{p^4}{\delta^2} \int_0^T \int \psi \chi^{2mp} |\nabla^m Ric|^{2p} \\
\end{split}
\]
where $V_r =Vol_{g(0)}(B_{g(0)}(q,r))$, the constant $C$ depends on $n,m,C_0$, and the last inequality follows from Young's inequality. Then the Sobolev inequality implies
\[
\begin{split}
\| \chi |\nabla^m Ric| \|_{L^\frac{2(n+2)p}{n} (r)}  \leq & C_S^\frac{n}{2p(n+2)} (\frac{p^4}{\delta^2} \max \{T, 1\}) ^\frac{1}{2p} \| \chi |\nabla^m Ric| \|_{L^{2p} (r+\delta)} \\
& + C_S^\frac{n}{2p(n+2)}  ( \frac{p^4}{\delta^2} \max\{T, 1\}) ^\frac{1}{2p} V_r^\frac{1}{2p}C_1 .\\
\end{split}
\]
Define
\[
F(r,p) = \max \{ \| \chi |\nabla^m Ric| \|_{L^{2p} (r)}, \hspace{2pt} V_r^\frac{1}{2p} C_1 \}.
\]
The rest of the proof is similar as for Lemma \ref{lemma: Modified Shi's estimate - mean value inequality form}.
\end{proof}

\begin{lem}\label{lemma: mean value inequality for derivative of ricci for short time}
If we assume $\chi^m|\nabla^m Rm|(t) \leq \frac{C_0}{\sqrt{t}}$ instead of the uniform bound in Lemma \ref{mean value inequality for derivatives of Ric}, and other conditions remain the same. Then the conclusion of Lemma \ref{mean value inequality for derivatives of Ric} still holds for $T\leq \tau_0$, where $\tau_0$ is a constant depending only on $n,m$ and $C_0$.
\end{lem}
\begin{proof}
In the proof of Lemma \ref{mean value inequality for derivatives of Ric}, the only place where we used the bound of $\chi^m |\nabla^m Rm|$ is in the estimate of $I_4$. Under the new assumption we have
\[
\begin{split}
I_4  \leq & C(n) \int \psi^2 \chi^{2mp+2} |Rm| |\nabla^m Ric|^{2p} \\
& + C(n) \sum_{i+j+k=m; \hspace{4pt} k\leq m-1} \int \psi^2 \chi^{2mp} |\nabla^i (\chi^2)| |\nabla^j Rm| |\nabla^k Ric| |\nabla^m Ric|^{2p-1} \\
 \leq & C(n,C_0) \int \psi^2 \chi^{2mp} |\nabla^m Ric|^{2p}  + C(n,m,C_0) C_1 t^{-\frac{1}{2}}\int \psi^2 \chi^{2mp-m}  |\nabla^m Ric|^{2p-1}.\\
\end{split}
\]
When integrated on $[0,T]$ we have
\[
\begin{split}
\int_0^T I_4  \leq &  C(n,C_0) \int_0^T \int \psi^2 \chi^{2mp} |\nabla^m Ric|^{2p}  \\
&+ C(n,m,C_0) C_1 \sqrt{T}\sup_{0<t \leq T}\int \psi^2 \chi^{2mp-m}  |\nabla^m Ric|^{2p-1}(t)\\
 \leq  & C(n,C_0) \int_0^T \int \psi^2 \chi^{2mp} |\nabla^m Ric|^{2p} + C(n,m,C_0, C_2)C_1^{2p} \sqrt{T} V_{r+\delta}.\\
 &+ C(n,m,C_0) \sqrt{T}\sup_{0<t \leq T}\int \psi^2 \chi^{2mp}  |\nabla^m Ric|^{2p}(t),
\end{split}
\]
where we used Young's inequality to get the last inequality.

If $C(n,m,C_0) \sqrt{T} < \frac{1}{2}$, then we can use the left hand side to absorb the last term and show that
\[
\begin{split}
& \frac{1}{2}\sup_{0 \leq t \leq T}\int \psi^2 \chi^{2mp} |\nabla^m Ric|^{2p} (t) + \frac{p}{2} \int_0^T \int \psi^2 \chi^{2mp} |\nabla ^m Ric|^{2p-2} |\nabla^{m+1} Ric|^2 \\
& + \frac{p(p-1)}{2} \int_0^T \int \psi^2 \chi^{2mp}|\nabla^m Ric|^{2p-4}|\nabla |\nabla^m Ric|^2|^2  \\
 \leq & C V_{r+\delta}C_1^{2p} + C \frac{p^4}{\delta^2} \int_0^T \int \psi \chi^{2mp} |\nabla^m Ric|^{2p}. \\
\end{split}
\]
Then the same arguments as before yields the lemma.
\end{proof}

Similar to Lemma \ref{lemma: L2 control of the derivatives of curvature}, we can control the $L^2$ norm of higher order covariant derivatives of the Ricci curvature tensor using lower order information.
\begin{lem}\label{L2 control of the derivatives of Ricci curvature}
Let $m\geq 1$, $R>0,T>0$. Suppose there are constant $C_0, C_1,$ such that
\[  \chi^{i} |\nabla^{i} Rm| (x,t) \leq C_0, \quad i=0,1,...,m-1;\]
\[ |\nabla^j \chi|_{g(t)}(x,t) \leq C_0, \quad j=1,2,...,m+1;\]
\[  \chi^{k} |\nabla^{k} Ric| (x,t) \leq C_1, \quad k=0,1,...,m-1;\]
\[Vol_{g(t)} B_{g(0)}(q,2R) \leq V;\]
for all  $(x,t)\in B_{g(0)}(q,2R) \times [0,T] $. Then
\[
\|\chi^m |\nabla^m Ric| \|_{L^2(B_{g(0)}(q,R) \times [0,T])} \leq C C_1 \sqrt{(1+\frac{T}{R^2} + T C_0)V},
\]
where $C$ depends on $n$ and $m$.
\end{lem}
\begin{proof}
Choose $C^1$ functions $0 \leq \psi \leq 1$, such that
\begin{equation*}
\psi(x) =\begin{cases}
1, \quad x\in B_{g(0)}(q,R); \\
0, \quad x\notin B_{g(0)}(q, 2R);
\end{cases}
\end{equation*}
and
\[ |\nabla \psi|^2_{g(0)} \leq \frac{4}{R^2} \psi\]
on its support. Using (\ref{evolution equation of the m^th covariant derivative of Ric}) and integration by parts we can calculate
\[
\begin{split}
&\frac{\partial}{\partial t} \int \psi^2 \chi^{2(m-1)} |\nabla^{m-1} Ric|^2  \\
=& \int \psi^2 \chi^{2(m-1)}\langle \nabla^{m-1} Ric, \chi^2 \Delta \nabla^{m-1} Ric\rangle \\
& +\sum_{i+j+k=m+1, \hspace{4pt} k\leq m}\int \psi^2 \chi^{2(m-1)}\langle \nabla^{m-1} Ric,  \nabla^i \chi \ast \nabla^j \chi \ast \nabla^k Ric  \rangle \\
 & + \sum_{i+j+k=m-1} \int \psi^2 \chi^{2(m-1)}\langle \nabla^{m-1} Ric,  \nabla^i(\chi^2) \ast \nabla^j Rm \ast \nabla^k Ric\rangle \\
\leq & - \int \psi^2 \chi^{2m} |\nabla^m Ric|^2 + \int |\nabla (\psi^2 \chi^{2m})| |\nabla^{m-1} Ric| |\nabla^m Ric| \\
        & + \int \psi^2 \chi^{2m-1} |\nabla \chi| |\nabla^{m-1} Ric| |\nabla^m Ric|  \\
        & + C(n,m) \sum_{i+j+k=m+1, \hspace{4pt} k\leq m-1}\int \psi^2 \chi^{2(m-1)}|\nabla^i \chi| | \nabla^j \chi|  |\nabla^k Ric | |\nabla^{m-1} Ric| \\
        & + C(n,m) \sum_{i+j+k=m-1} \int \psi^2 \chi^{2(m-1)}  |\nabla^i(\chi^2)| |\nabla^j Rm|  |\nabla^k Ric| |\nabla^{m-1} Ric|. \\
\end{split}
\]
By the Cauchy inequality
\[
\frac{\partial}{\partial t} \int \psi^2 \chi^{2(m-1)} |\nabla^{m-1} Ric|^2  \leq  - \frac{1}{4}\int \psi^2 \chi^{2m} |\nabla^m Ric|^2 +I,
\]
where
\[
\begin{split}
I =&  \frac{8}{R^2}\int \psi \chi^2 \chi^{2(m-1)} |\nabla^{m-1} Ric|^2 + 2\int \psi^2 | \nabla \chi|^2 \chi^{2(m-1)} |\nabla^{m-1} Ric|^2 \\
        & + \int \psi^2 \chi^{2m-2} |\nabla \chi|^2 |\nabla^{m-1} Ric|^2 \\
        & + C(n,m) \sum_{i+j+k=m+1, \hspace{4pt} k\leq m-1}\int \psi^2 \chi^{2(m-1)}|\nabla^i \chi| | \nabla^j \chi|  |\nabla^k Ric | |\nabla^{m-1} Ric| \\
        & + C(n,m) \sum_{i+j+k=m-1} \int \psi^2 \chi^{2(m-1)}  |\nabla^i(\chi^2)| |\nabla^j Rm|  |\nabla^k Ric| |\nabla^{m-1} Ric|. \\
\end{split}
\]
Integrate on $[0,T]$, we get
\[
\begin{split}
\frac{1}{4} \int_0^T \int \psi^2 \chi^{2m} |\nabla^m Ric|^2 \leq & \int \psi^2 \chi^{2(m-1)} |\nabla^{m-1} Ric|^2 (0) + \int_0^T I dt \\
\leq & V C_1^2 + VT C(n,m)(\frac{1}{R^2}+C_0) C_1^2.
\end{split}
\]
\end{proof}

\subsection{Lifespan estimate for the local Ricci flow}\label{subsection: lifespan estimate for the local ricci flow}
\begin{prop}\label{theorem: lifespan estimate - local ricci flow}
Let $(M,g)$ be an $n$-dimensional complete Riemannian manifold satisfying Assumption \ref{initial good cover assumption}, and let $g(t)$ be a solution of the local Ricci flow $\frac{\partial}{\partial t}g=-2\chi^2 Ric$ starting from $(M,g)$, where the smooth function $0 \leq \chi \leq 1$ is compactly supported with $|\nabla_{g(0)}^m \chi|(x)\leq C(n,I)K_i^{1+m/2}$ for $x \in \Bhat_i$, $m=1,2,3$, $i=1,2,...,N$. Then the lifespan of $g(t)$ is greater than a constant $T_0$ depending only on $n, A, \bar{K}$ and $I$. More precisely, $T_0$ bounds from below the following generalized doubling time in Definition \ref{definition: generalized doubling time}.
\end{prop}

\begin{defn}\label{definition: generalized doubling time}
We define the generalized doubling time $T$ to be the largest number such that
\[ \sup_{\Bhat_i \times [0,T]} |Rm| \leq 2(1+\Gamma)K_i,\]
\[ \int_0^T \sup_{\Bhat_i} |Ric| \leq \frac{1}{2} \ln 2,\]
\[ \int_0^T \sup_{\Bhat_i } \chi |\nabla Ric| \leq \sqrt{K_i},\]
\[ \int_0^T \sup_{\Bhat_i } \chi^2 |\nabla^2 Ric| \leq K_i,\]
for $i=1,2,...,N$.
\end{defn}
By Lemma \ref{lemma: control nabla Rm by Rm and nabla nabla Ric} we have $|\nabla Rm|_{g(0)}(x)\leq C(n)K_i^{3/2}$ for $x\in B_{g(0)}(x_i, 15r_i) \subset \Bhat_i$. Hence by Lemma \ref{lemma: cut-off function on each ball} we can construct a cut-off function $\phi_i$ on each $\Bhat_i$, which equals $1$ on $B_i$, and has three bounded derivatives:
\[|\nabla^m \phi_i|_{g(0)} \leq C(n) K_i^\frac{m}{2}, \quad m=1,2,3.\]
\begin{lem}\label{lemma: derivative estimats of chi}
For any $x \in \Bhat_i$ and $ t\in [0,T]$, we have
\[|\nabla^m \chi|(x,t) \leq C(n, I) K_i^\frac{m}{2}, \quad m=1,2,3,\]
and
\[| \nabla^2 \phi_i|(x,t) \leq C(n) K_i .\]
\end{lem}
\begin{proof}
For any $x \in \Bhat_i$, by the assumption of Proposition \ref{theorem: lifespan estimate - local ricci flow},
\[|\nabla^m \chi|(x,0) \leq C(n,I) K_i^\frac{m}{2}, \quad m=1,2,3.\]
From (\ref{equation: inequality satisfied by the derivatives of chi}) we get
\[\frac{\partial}{\partial t} |\nabla \chi| \leq C(n) \chi^2 |Ric| |\nabla \chi| ,\]
\[\frac{\partial}{\partial t} |\nabla^2 \chi| \leq C(n) \chi^2 |\nabla Ric| | \nabla \chi| +  C(n) \chi   |Ric| | \nabla \chi|^2 + \chi^2 |Ric| |\nabla^2 \chi|,\]
\[\begin{split}
\frac{\partial }{\partial t} |\nabla^3 \chi | \leq & C(n) \left(\chi^2 |\nabla^2 Ric| | \nabla \chi|  + |Ric| | \nabla \chi|^3 + \chi^2 |\nabla Ric| | \nabla^2 \chi | \right)\\
& + C(n) \left( \chi |\nabla Ric| | \nabla \chi |^2 + \chi |Ric| |\nabla \chi | | \nabla^2 \chi |+ \chi^2 |Ric| |\nabla^3 \chi| \right). \\
\end{split}\]
Then the lemma follows from Gronwall's inequality and the integral bounds in Definition \ref{definition: generalized doubling time}.

$|\nabla^2 \phi_i|$ is estimated similarly.
\end{proof}

By Definition \ref{definition: generalized doubling time}, we have uniformly equivalent metrics on the time interval $[0,T]$:
\begin{equation}\label{equivalence of metric}
\frac{1}{2}g(x,0) \leq g(x,t) \leq 2g(x,0), \quad x \in M, \hspace{2pt} t\in [0,T].
\end{equation}

Now we can apply Lemma \ref{lemma: Shi's estimate for the local ricci flow} and Lemma \ref{lemma: Modified Shi's estimate for the local Ricci flow}, which play the roles of Shi's estimate and modified Shi's estimate for the Ricci flow, to show the following derivative estimates for $Rm$ on each $B_i$.
\begin{lem}\label{lemma: estimate of Rm and its covariant derivatives under local ricci flow}
There is a constant $C_0$ depending only on $n, \Gamma$ and $I$, such that

(i) $\chi|\nabla Rm|(x,t) \leq C_0K_i^{3/2}$, for all $(x,t)\in B_i \times [0,T]$;

(ii) $\chi^2|\nabla^2 Rm|(x,t) \leq \frac{C_0K_i^{3/2}}{\sqrt{t}}$ for $(x,t)\in B_i \times [0,\frac{1}{K_i}]$, and $\chi^2|\nabla^2 Rm|(x,t)\leq C_0 K_i^{2}$ for $(x,t)\in B_i \times[\frac{1}{K_i}, T]$.
\end{lem}
\begin{proof}
By Lemma \ref{lemma: control nabla Rm by Rm and nabla nabla Ric}, $|\nabla Rm|$ is initially bounded. Then by Lemma \ref{lemma: Modified Shi's estimate for the local Ricci flow} and \ref{lemma: derivative estimats of chi} we have
\[ \chi|\nabla Rm| (x,t) \leq C(n, I, \Gamma) K_i^{3/2}\]
for $(x,t) \in B_{g(0)}(x_i,4r_i) \times [0, \frac{1}{K_i}]$.

By the equivalence of metrics (\ref{equivalence of metric}), $B_{g(t)}(x_i, 8r_i) \subset \Bhat_i$ for all $t\in [0,T]$.
Therefore $|Rm| \leq 2(1+\Gamma)K_i$ on $B_{g(\tau)}(x_i, 8r_i)$ for all $\tau \in [0,T]$. Hence by Lemma \ref{lemma: Shi's estimate for the local ricci flow}, we have
\[\chi |\nabla Rm|(x,t) \leq \frac{C(n, I, \Gamma)K_i}{(t-\tau)^{1/2}}, \]
for $(x,t) \in B_  {g(\tau)}(x_i, 6 r_i) \times (\tau, \tau + \frac{1}{K_i}]$, where $\tau \in [0, T-\frac{1}{K_i}]$. In particular,
\[\chi |\nabla Rm|(x,t) \leq C(n, I, \Gamma)K_i^{3/2},\]
for all $t \in [\frac{1}{K_i}, T]$ and $x\in B_{g(0)}(x_i,3r_i)$. Now we have verified (i). We can prove (ii) similarly using Lemma \ref{lemma: Shi's estimate for the local ricci flow}.
\end{proof}
Then we can prove the following estimates of the Ricci curvature and its first two covariant derivatives.
\begin{lem}\label{lemma: estimates of Ricci curvature and its covariant derivatives - local ricci flow}
Suppose the initial manifold $(M,g)$ satisfies Assumption \ref{initial good cover assumption}, then there are constants $C$ and $\Lambda$ depending on $n, \Gamma, I$ and $A$, such that if $T \leq \frac{A}{ \Lambda}$, then
\[\chi^m|\nabla^m Ric|(x,t) \leq CK_i^\frac{m}{2} P_ie^{\Lambda K_i t}, \quad m=0,1,2,\]
for $(x,t) \in B_i \times [0, T]$, $i=1,2,...,N$.
\end{lem}
\begin{proof}
STEP 1: Control $|Ric|$ and $|\nabla Ric|$.

By (\ref{evolution equation of the m^th covariant derivative of Ric}), the Ricci curvature tensor satisfies
\[
\begin{split}
\frac{\partial}{\partial t } Ric = \chi^2 \Delta Ric + \chi \nabla \chi \ast \nabla Ric + \chi \nabla^2 \chi \ast Ric + \nabla \chi \ast \nabla \chi \ast Ric + \chi^2 Rm \ast Ric.
 \end{split}
\]
Hence
\[
\begin{split}
\frac{\partial}{\partial t } |Ric|^2 \leq & \chi^2 \Delta |Ric|^2 -2 \chi^2 |\nabla Ric|^2 + C(n) \chi |\nabla \chi| |\nabla Ric||Ric| + C(n) \chi |\nabla^2 \chi| |Ric|^2\\
& + C(n) |\nabla \chi|^2 |Ric| + \chi^2 |Rm| |Ric|^2 + 4\chi^2 |Ric|^3.
\end{split}
\]
By Lemma \ref{lemma: derivative estimats of chi}, we have
\[
\begin{split}
\frac{\partial }{\partial t} |Ric|^2 \leq \chi^2 \Delta |Ric|^2 + C_1 K_i |Ric|^2,
\end{split}
\]
which holds on $B_i$ for some constant $C_2$ depending on $n, \Gamma$ and $I$.

Let $\Omega$ be the support of $\chi$, and assume that $\Omega$ is covered by $B_i, i=1,2,...,N$. Define
\[\Phi (x,t) =\sum_{i=1}^N \phi_i(x) P^2_i \exp(\Lambda_0 K_i t),\]
where $\Lambda_0$ is a constant to be determined, $\phi_i$ are the cut-off functions in Lemma \ref{lemma: derivative estimats of chi}.

Recall that $P_i=(n-1)K_ie^{A\bar{K}-AK_i}$. Since $1 \leq  K_j \leq K_i + \Gamma$, it is easy to check that $P_j \leq (1+\Gamma)e^{A\Gamma} P_i$ when $\Bhat_j \cap \Bhat_i \neq \emptyset$.

If $T \leq \frac{A}{\Lambda_0}$, then for any $(x,t) \in B_i \times [0,T]$ we have the following pointwise estimates.
\begin{eqnarray*}
\Phi(x,t) & = & \sum_{\Bhat_j \cap B_i \neq \emptyset} \phi_j P^2_j \exp(\Lambda_0 K_j t) \\
          & \leq  & \sum_{\Bhat_j \cap B_i \neq \emptyset} (1+\Gamma)^2 e^{2A\Gamma} P^2_i e^{\Lambda_0 \Gamma t}\exp(\Lambda_0 K_i t) \\
          & \leq & C_3 I P^2_i \exp(\Lambda_0 K_i t),
\end{eqnarray*}
where $C_3=(1+\Gamma)^2e^{3A\Gamma}$. By Lemma \ref{lemma: derivative estimats of chi} we have $|\Delta \phi_i|(x,t) \leq C_2 K_i $ for some $C_2$ depending only on $n$, hence
\begin{eqnarray*}
 \Delta \Phi(x,t) & = & \sum_{\Bhat_j \cap B_i \neq \emptyset} \Delta \phi_j(x,t) P^2_j \exp(\Lambda_0 K_j t) \\
                  & \leq & \sum_{\Bhat_j \cap B_i \neq \emptyset} C_2 K_j P^2_j \exp(\Lambda_0 K_j t) \\
                  & \leq & \sum_{\Bhat_j \cap B_i \neq \emptyset} C_2 ( K_i+ \Gamma ) P^2_j \exp(\Lambda_0 (K_i+\Gamma) t) \\
                  & \leq & (1+\Gamma)^3IC_2e^{3A\Gamma} K_i P^2_i \exp(\Lambda_0 K_i t).
\end{eqnarray*}

\begin{eqnarray*}
 \frac{\partial}{\partial t} \Phi(x,t) & = & \sum_{\Bhat_j \cap B_i \neq \emptyset} \phi_j(x,t) P^2_i \Lambda_0 K_i \exp(\Lambda_0 K_i t) \\
                                       & \geq & \Lambda_0 K_i P^2_i \exp(\Lambda_0 K_i t)
\end{eqnarray*}

Take
\[\Lambda_0 \geq C_1 C_3 I + (1+\Gamma)^3IC_2e^{3A\Gamma},\]
then
\begin{equation}
(\frac{\partial}{\partial t} - \chi^2 \Delta ) \Phi(x,t)  \geq C_1 C_3I K_i P^2_i \exp(\Lambda_0 K_i t) \geq  C_1 K_i \Phi(x,t).
\end{equation}

Therefore

\[
 (\frac{\partial}{\partial t } - \chi^2 \Delta) (|Ric|^2- \Phi)(x,t)  \leq  C_1 K_i (|Ric|^2 -  \Phi)(x,t).
\]

Since $B_i, i=1,2,..., N < \infty$ cover $\Omega$, we have
\[(\frac{\partial}{\partial t } - \chi^2 \Delta) (|Ric|^2- \Phi) \leq C_1 (\sum_i^N \phi_i K_i) (|Ric|^2 - \Phi), \]
holds on $\Omega$.

By Assumption \ref{initial good cover assumption},
\[|Ric|^2(x,0) \leq \Phi(x,0), \quad x\in \Omega.\]
And since $\chi$ is compactly supported on $\Omega$, the metric does not change near the boundary and outside of $\Omega$, hence
\[|Ric|^2(x,t) = |Ric|^2(x,0) \leq \Phi(x,0) \leq \Phi(x,t), \quad  x\in \partial \Omega, \quad t \in [0,T].\]
Hence the maximum principle implies
\[|Ric|^2(x,t) \leq \Phi(x,t),\]
for all $(x,t) \in M \times [0,T]$. In particular, when $(x,t) \in B_i \times [0,T]$, we have
\[|Ric|(x,t) \leq \sqrt{C_3 I } P_i \exp(\frac{1}{2}\Lambda_0 K_i t). \]
And when $x \in \Bhat_i $ it must be covered by some $B_j$ with $\Bhat_j \cap \Bhat_i \neq \emptyset$, so
\[|Ric|(x,t) \leq \sqrt{C_3 I } (1+\Gamma)e^{2A\Gamma} P_i \exp(\frac{1}{2}\Lambda_0 K_i t), \]
as long as $t \leq T$.

Similarly we have a differential inequality for the first covariant derivative of the Ricci curvature tensor on $B_i$:
\[\frac{\partial}{\partial t} |\nabla Ric|^2  \leq \chi^2 \Delta |\nabla Ric|^2 + C_4 K_i |\nabla Ric|^2  + C_4 K_i^2 |Ric|^2,\]
where $C_4$ depends on $n, I$ and $\Gamma$.

Let
\[\Phi_1(x,t)= \sum_{i=1}^N \phi_i K_i P^2_i \exp(\Lambda_1 K_i t),\]
as before we can find a constant $\Lambda_1 = C(n,\Gamma, I) I e^{3A\Gamma}$ (WLOG we choose $\Lambda_1\geq \Lambda_0$), such that if $T \leq \frac{A}{ \Lambda_1}$ then we have
\[
(\frac{\partial}{\partial t} - \chi^2 \Delta ) \Phi_1(x,t)  \geq  C_4 K_i \Phi_1(x,t) + C_4 K_i^2 \Phi(x,t),
\]
for all $(x,t)\in B_i \times [0,T]$. Hence the same argument for controlling $|Ric|$ yields
\[|\nabla Ric|^2 (x,t) \leq \Phi_1(x,t).\]
In particular,
\[|\nabla Ric| (x,t) \leq C(n,\Gamma, I) e^{2A\Gamma} \sqrt{I}  K_i^\frac{1}{2} P_i e^{\frac{1}{2}\Lambda_1 K_i t},\]
when $(x,t) \in B_i \times [0,T]$. And
\[|\nabla Ric| (x,t) \leq C(n,\Gamma, I) e^{4A\Gamma} \sqrt{I}  K_i^\frac{1}{2} P_i e^{\frac{1}{2}\Lambda_1 K_i t},\]
when $(x,t) \in \Bhat_i \times [0,T]$.

\smallskip

STEP 2: Control $\chi^2 |\nabla ^2 Ric|$.

Lemma \ref{lemma: estimate of Rm and its covariant derivatives under local ricci flow} provides control of $|\nabla Rm|$ and $|\nabla^2 Rm|$,  STEP 1 provides control of $|Ric|$ and $|\nabla Ric|$ in $\Bhat_i \times [0,T]$. The integral bound on $|Ric|$ in Definition \ref{definition: generalized doubling time} implies
\[Vol_{g(t)} B_{g(0)} (x_i, r) \leq C(n) V(r), \quad  t\in [0,T],\]
where $V(r)=Vol_{g(0)} B_{g(0)}(x_i, r)$, $r < 16 r_i$.

WLOG we can assume $r_i \leq 2/\sqrt{K_i}$, otherwise we can cover $B_i$ with finitely many balls with radius $2/\sqrt{K_i}$ and work on each smaller ball. Then Lemma \ref{lemma: Sobolev inequality} and the equivalence of metrics (\ref{equivalence of metric}) imply a uniform Sobolev inequality on the time interval $[0,T]$:
\[
\left( \int |f|^\frac{2n}{n-2} d\mu(t)\right)^\frac{n-2}{n} \leq C(n) \frac{r_i^2}{Vol_{g(0)}(\Bhat_i)} \int (|\nabla f|_{g(t)}^2 + r_i^{-2} f^2) d\mu(t),
\]
for $f$ compactly supported on $\Bhat_i$.

We can use Lemma \ref{L2 control of the derivatives of Ricci curvature} to prove
\[ \| \chi^2 |\nabla^2 Ric| \|_{L^2(B_{g(0)}(x_i, 2r_i) \times [0,t])} \leq C(n,\Gamma, I)e^{4A\Gamma}\sqrt{(1+T)V(2r_i)} K_i P_i \exp( \Lambda_1 K_i t).\]
Then Lemma \ref{mean value inequality for derivatives of Ric} and \ref{lemma: mean value inequality for derivative of ricci for short time} imply that
\[\sup_{B_i \times [0, t]} \chi^2 |\nabla^2 Ric| \leq C(n, \Gamma,I) e^{4A\Gamma}\max\{A,1\}^2 K_i P_i \exp( \Lambda_1 K_i t),\]
where  $0 \leq t \leq T \leq \frac{A}{ \Lambda_1}$.
\end{proof}
\begin{rem}\label{remark: handling A by e^{-A}}
$\max\{A,1\}^2$ in the above estimate can be absorbed by the exponential term since $\max\{A,1\}^2 e^{-AK_i}\leq 4e^{-AK_i/2}$.
\end{rem}

Now we can prove the main proposition of this subsection.
\begin{proof}[Proof of Proposition \ref{theorem: lifespan estimate - local ricci flow}]

Under the local Ricci flow, the Riemann curvature tensor evolves by
\[
\begin{split}
\frac{\partial}{\partial t} R_{ijk}^l =& -\nabla_i\nabla_k (\chi^2 R_j^l) + \nabla_i\nabla^l (\chi^2 R_{jk}) +\nabla_j\nabla_k (\chi^2 R_i^l) -\nabla_j\nabla^l (\chi^2 R_{ik}) \\
&+ \chi^2 g^{lp} ( R_{ijk}^q R_{pq} + R_{ijp}^q R_{kq}).
\end{split}
\]
By Lemma \ref{lemma: derivative estimats of chi}, at a point $(x,t) \in B_i \times [0,T]$ we have
\begin{equation}\label{equation: inequality satisfied by Rm in the proof of lifespan of local Ricci flow}
\begin{split}
\frac{\partial}{\partial t} |Rm| \leq & 8 ( \chi |\nabla^2 \chi| + |\nabla \chi|^2 ) |Ric| + 8 \chi |\nabla \chi| |\nabla Ric| + 4 \chi^2 |\nabla^2 Ric|  \\
&+ 4 \chi^2 |Rm| |Ric| \\
\leq & C_1 (K_i |Ric| + \sqrt{K_i} \chi |\nabla Ric| + \chi^2 |\nabla^2 Ric| ),
\end{split}
\end{equation}
where $C_1$ depends on $n, I$.

Suppose $x\in B_i$ and $t\leq T \leq \frac{A}{ \Lambda_1 }$, by Lemma \ref{lemma: estimates of Ricci curvature and its covariant derivatives - local ricci flow} we have
\[
\begin{split}
\int_0^t \chi^m |\nabla^m Ric|(x,\sigma) d\sigma \leq & \int_0^T C K_i^{m/2} P_i \exp(\Lambda_1 K_i \sigma) d\sigma \\
= & C K_i^{m/2} e^{A\bar{K}-AK_i}\Lambda_1^{-1} (e^{\Lambda_1 K_i T} -1),
\end{split}
\]
where $m=0,1,2$.
By Lemma \ref{lemma: minimum of a continuous function}, $f(s)=h(\frac{ \Lambda_1 \epsilon  }{ Ce^{A\bar{K}}}, A) (s)$ has a minimum $f_{min} > 0$, which can be estimated from below by a constant in the form $C(n, I, \Gamma)e^{-A\bar{K}-4A\Gamma}A\epsilon$.

If $T < f_{min}/\Lambda_1$, then direct calculation yields
\begin{equation}\label{equation: integral bounds on Ric and its derivatives}
\int_0^t \chi^m |\nabla^m Ric|(x,\sigma) d\sigma < \epsilon K_i^{m/2},\quad m=0,1,2,
\end{equation}
and
\[
\begin{split}
\sup_{B_i}|Rm|(t)  & \leq \sup_{B_i} |Rm|(0) + C_1\int_0^t (K_i  |Ric| + \sqrt{K_i} \chi |\nabla Ric| + \chi^2 |\nabla^2 Ric| ) \\
& \leq (1+3C_1 \epsilon) K_i,
\end{split}\]
when $t \leq T$.

For any $x\in \Bhat_i$, it must be covered by some $B_j$ such that $\Bhat_j \cap \Bhat_i \neq \emptyset$, hence the above inequalities still hold if we multiple the right hand sides by $1+\Gamma$. Therefore, if we take
\[\epsilon = \min\{\frac{1}{2}\ln 2, \frac{1}{9C_1}, \frac{1}{1+\Gamma}\},\]
then the integral bounds in Definition \ref{definition: generalized doubling time} are satisfied with sharp inequalities. Therefore we must have
$T \geq T_0 = f_{min}/\Lambda_1$. Keeping track of the constants, we see that $T_0=C(n,I,\Gamma)e^{-A\bar{K}-7A\Gamma}A$.
\end{proof}

\subsection{Existence of Ricci flows on noncompact manifolds}\label{subsection: existence of rf}

We will first prove the existence of Ricci flow solutions on manifolds satisfying Assumption \ref{initial good cover assumption}.

\begin{thm}\label{theorem: existence of ricci flow on noncompact manifolds}
Let $(M,g)$ be an $n$-dimensional complete noncompact Riemannian manifold satisfying Assumption \ref{initial good cover assumption}. Then the Ricci flow with initial metric $g$ has a solution on $M \times [0,T_0]$, where $T_0 >0$ depends only on $n,\Gamma, A, \bar{K}, I$. Moreover, for any $t\in [0,T_0]$, we have
\[\frac{1}{2}g(0) \leq g(t) \leq 2g(0),\]
\[\sup_{\Bhat_i} |Rm|_{g(t)} \leq 2(1+\Gamma)K_i, \hspace{4pt} i=1,2,3,...\]
\end{thm}
\begin{rem}\label{remark: preserve unboundedness of curvature}
In the above theorem, we can choose $T_0$ smaller, depending on the value of $\Gamma$, such that if $\sup_{\Bhat_i}  |Rm|_{g(0)} = K_i$ is achieved for an index $i$, then we also have
\[\sup_{\Bhat_i \times [0,T_0] }|Rm| \geq \frac{2}{3} K_i.\]
\end{rem}

Let's first construct the cut-off functions used in the local Ricci flow.
\begin{lem}\label{lemma: construction of chi}
Suppose $(M,g)$ satisfies Assumption \ref{initial good cover assumption}, and $\Omega$ is a compact domain in $M$. Then we can construct a smooth function $0 \leq \chi \leq 1$, which is compactly supported, and $\chi \equiv 1$ on $\Omega$. Moreover, for each index $i$ we have
\[
|\nabla^k \chi|(x) \leq C(n, I, \Gamma) K_{i}^\frac{k}{2}, \quad k=1,2,3, \quad when \quad x\in \Bhat_i.
\]
\end{lem}
\begin{proof}
We can find a compact domain $\hat{\Omega}$ containing $\Omega$, such that, with possibly reordering the geodesic balls in the cover, we have two integers $N_1 < N_2$, such that

1. $\Omega \subset \cup_{i=1}^{N_1} B_i$, and $\cup_{i=1}^{N_1} \Bhat_i \subset \hat{\Omega}$;

2. $\hat{\Omega} \subset \cup_{i=1}^{N_2} B_i$, and $\partial \hat{\Omega} \subset \cup_{i=N_1+1}^{N_2} \Bhat_i$;

3. $ \Omega \cap \cup_{i=N_1+1}^{N_2} \Bhat_i =\emptyset$.

For each $\Bhat_i$ in the cover, let $\phi_i$ be the cut-off function on $\Bhat_i$ defined in Lemma \ref{lemma: cut-off function on each ball}, which equals $1$ in $B_i$. Define
\begin{equation}
\chi=\frac{\sum_{i=1}^{N_1} \phi_i}{\sum_{j=1}^{N_2} \phi_j}
\end{equation}
on $\hat{\Omega}$, and define it to be $0$ outside of $\hat{\Omega}$.
Then it is easy to check that $\chi$ has the desired properties.
\end{proof}

Now let's prove the theorem.
\begin{proof}[Proof of Theorem \ref{theorem: existence of ricci flow on noncompact manifolds}]
Let $\Omega_j, j=1,2,...,$ be a compact exhaustion of $M$, for each $\Omega_j$, let $\chi_j$ be the cut-off function constructed in Lemma \ref{lemma: construction of chi}. Existence of the local Ricci flow
\[\frac{\partial}{\partial t} g = -2 \chi_j^2 Ric, \quad g(0)=g\]
is due to D. Yang \cite{yang1992convergence}, see also G. Xu \cite{xu2013short} for a rigorous proof. By Proposition \ref{theorem: lifespan estimate - local ricci flow} the solution $g_j(t)$ exists on $[0, T_0]$, where $T_0>0$ is independent of $j$. And $g_j(t)$ are uniformly equivalent to the initial metric $g$, with curvature uniformly bounded on any compact subset of $M$.

On $\Omega_j$, $g_j(t)$ agrees with the Ricci flow. Hence, by passing to subsequences when necessary, $g_j(t)$ converges to a Ricci flow solution $g(t)$ on $M \times [0, T_0]$ in Cheeger-Gromov sense. The equivalence of metrics $g(t), t\in [0,T]$ and the curvature bound are clear from Definition \ref{definition: generalized doubling time}.
\end{proof}

Moreover, if we take $\epsilon < \frac{1}{9C_1(1+\Gamma)}$ in (\ref{equation: integral bounds on Ric and its derivatives}) in the proof of Proposition \ref{theorem: lifespan estimate - local ricci flow}, then the differential inequality (\ref{equation: inequality satisfied by Rm in the proof of lifespan of local Ricci flow}) implies that
\[
\begin{split}
|Rm|(x,t) \geq & |Rm|(x,0) -  C_1\int_0^t (K_j  |Ric| + \sqrt{K_j} \chi |\nabla Ric| + \chi^2 |\nabla^2 Ric| ) \\
              \geq &  |Rm|(x,0) - \frac{1}{3(1+\Gamma)} K_j,
\end{split}
\]
for all $(x,t) \in B_j \times [0,T_0]$. Since $K_j \leq K_i+\Gamma \leq (1+\Gamma)K_i$ whenever $\Bhat_j \cap \Bhat_i \neq \emptyset$, we have
\[ |Rm|(x,t) \geq |Rm|(x,0) -  \frac{1}{3} K_i\]
for all $(x,t) \in \Bhat_i$. If there is a point $x_0 \in \Bhat_i$ such that $|Rm|(x_0, 0) = K_i$, then
\[\sup_{\Bhat_i \times [0,T_0] }|Rm| \geq \frac{2}{3} K_i.\]
Hence we have Remark \ref{remark: preserve unboundedness of curvature}.

By the construction of the solution in Theorem \ref{theorem: existence of ricci flow on noncompact manifolds} and the proof of Proposition \ref{theorem: lifespan estimate - local ricci flow}, we have the following lemma which will be useful later.
\begin{lem}\label{lemma: barrier estimates of Ric and its derivatives}
Under a Ricci flow solution constructed by Theorem \ref{theorem: existence of ricci flow on noncompact manifolds}, we have
\[ |\nabla^m Ric|^2 \leq \Phi_m, \quad m=0,1,2, \]
where
\[\Phi_m(x,t) = \sum_{i=1}^\infty \phi_i(x) K_i^m P_i^2 e^{\Lambda K_i t}, \quad m=0,1,2,\]
$\phi_i$ is a cut-off function on $\Bhat_i$ which equals $1$ in $B_i$, $i=1,2,3,...$, and $\Lambda$ is a positive constant depending on $n, A, I, \Gamma$.
\end{lem}

Now, as a corollary of Theorem \ref{theorem: existence of ricci flow on noncompact manifolds}, we can prove Theorem \ref{theorem: existence of ricci flow strong assumption}. Let's first state a lemma which is a direct consequence of Definition \ref{definition: second order curvature scale}.

\begin{lem}\label{lemma: no mutual inclusion under curvature scale}
Let $\rho_x$ be as in Definition \ref{definition: second order curvature scale}, for any two points $x \neq y$ in a Riemannian manifold $(M,g)$, $B_g(x,\rho_x)$ is not properly contained in $B_g(y, \rho_y)$ and vice versa.
\end{lem}
\begin{proof}
Suppose $B_g(x,\rho_x) \subset B_g(y,\rho_y)$ properly, then $\rho_x < \rho_y$ and we have
\[ \sup_{B_g(x,\rho_x)} |Rm| \leq \frac{1}{\rho_y^2} < \frac{1}{\rho_x^2},\]
\[ \sup_{B_g(x,\rho_x)} |\nabla ^m Ric| \leq \frac{n-1}{\rho_y^{2+m}} < \frac{n-1}{\rho_x^{2+m}}, \quad m=1,2.\]
However, by definition of $\rho_x$, we must have a point $x_0 \in \overline{B_g(x,\rho_x)}$, such that either
\[|Rm|(x_0) = \frac{1}{\rho_x^2} \quad or \quad |\nabla Ric|(x_0) =\frac{n-1}{\rho_x^{2+m}} \quad for \quad m=1 \quad or \quad 2,\]
which leads to a contradiction.
\end{proof}

\begin{proof}[Proof of Theorem \ref{theorem: existence of ricci flow strong assumption}]
We only need to check that the initial manifold in the theorem satisfies Assumption \ref{initial good cover assumption}, then the result follows from Theorem \ref{theorem: existence of ricci flow on noncompact manifolds}.

Let $\Omega$ be a compact domain on $M$. Let $L=17 \sqrt{1+\gamma}$, and let $B_g(x_i, r_i/L)$, $i=1,2,...N$, be a maximal set of mutually disjoint geodesic balls with $x_i \in \Omega$ and $r_i=\min \{\frac{\rho_{x_i}}{16}, 1\}$. Such a maximal set must exist since there is no local blow-up of curvature. For convenience we denote $B_i=B_g(x_i, r_i)$ and $\Bhat_i=B_g(x_i, 16r_i)$.

\emph{Claim 1:}  When $\Bhat_i\cap \Bhat_j \neq \emptyset$, we have
\[\frac{r_j}{16\sqrt{1+\gamma}} < r_i < 16\sqrt{1+\gamma} r_j.\]

\emph{Proof of Claim 1: } If $\Bhat_i\cap \Bhat_j \neq \emptyset$, assumption (ii) implies $|\rho_{x_i}^2 -\rho_{x_j}^2|\leq \gamma$, hence
\[\rho_{x_j}\geq \frac{\rho_{x_i}}{\sqrt{1+\gamma\rho_{x_i}^2}}.\]
The right-hand side of the above inequality is nondecreasing in $\rho_{x_i}$, thus when $\rho_{x_i} \geq 1$ we have
\[\rho_{x_j} \geq \frac{1}{\sqrt{1+\gamma}},\]
hence
\[r_i \leq 1 \leq 16\sqrt{1+\gamma} r_j.\]
On the other hand, if $\rho_{x_i}<1$, then we have
\[\rho_{x_i} \leq \rho_{x_j}\sqrt{1+\gamma \rho_{x_i}^2} < \rho_{x_j}\sqrt{1+\gamma},\]
hence
\[r_i = \frac{\rho_{x_i}}{16} \leq \sqrt{1+\gamma} r_j.\]
Therefore Claim 1 is proved.

Now we claim that $\Omega$ is covered by $B_i=B_g(x_i, r_i)$, $i=1,2,...,N$. Indeed, for any $y\in \Omega$, denote $r_y=\min \{\frac{\rho_y}{16}, 1\}$, by the maximality, there must be some index $i$ such that $B_g(y, r_y/L) \cap B_g(x_i, r_i/L) \neq \emptyset$. Note that the proof of Claim 1 also yields
\[ r_y \leq 16 \sqrt{1+\gamma} r_i. \]
Thus
\[d(x_i, y) \leq (r_i + r_y)/L < r_i,\]
and consequently $y\in B_i$. Therefore the claim is proved.

Choose a compact exhaustion of $M$ by $\Omega_1 \subset \Omega_2 \subset ...\subset \Omega_i \subset ...$, and construct a cover of each $\Omega _i$ as above, then we can take a subsequential limit to find a cover of $M$ by $B_i=B_g(x_i, r_i)$, $i=1,2,...$ with $B_g(x_i, r_i/L)$ disjoint. Hence we have verified (a).

Let $K_i=\max \{(1+\gamma)\bar{r}^2/r_i^2, 1\}$, where $\bar{r}$ is the same constant as in Assumption \ref{initial good cover assumption}, then
\[\sup_{\Bhat_i} |Rm| \leq K_i.\]
Let $y$ be any point in $\Bhat_i$. By Lemma \ref{lemma: no mutual inclusion under curvature scale} we have
\[ \rho_y < 2 \rho_{x_i}.\]
If $\rho_{x_i} < 1$, then $K_i = 256 (1+\gamma) \bar{r}^2/ \rho_{x_i}^2$, and assumption (ii) implies
\[\rho_{x_i}\leq \rho_y \sqrt{1+\gamma \rho_{x_i}^2} < \rho_y\sqrt{1+\gamma}.\]
Hence for $m=0,1,2$, by assumption (i) we have
\[ |\nabla^m Ric|(y) \leq \frac{(1+\gamma)^{1+m/2}(n-1)}{\rho_{x_i}^{2+m}} e^{\beta -\alpha \rho_{x_i}^{-2}/4 } \leq (n-1)K_i^{1+m/2} e^{\beta - \alpha K_i /1024(1+\gamma)\bar{r}^2}.\]
If $\rho_{x_i} \geq 1$, then $K_i \leq 256 (1+\gamma)\bar{r}^2$, and
\[ |\nabla^m Ric|(y) \leq (n-1)K_i^{1+m/2} \leq (n-1)K_i^{1+m/2} e^{\alpha/4 -\alpha K_i/1024(1+\gamma)\bar{r}^2}.\]
Therefore we can take $A=\frac{\alpha}{1024(1+\gamma)\bar{r}^2}$ and $\bar{K}=\max\{\frac{\beta}{A}, \frac{\alpha}{4A} \}$, then (b) and (c) are verified.

Now take $\Gamma=256(1+\gamma)\gamma\bar{r}^2$, then (d) directly follows from assumption (ii).

To check (e), we need to compare $r_i$ and $r_j$ when $\Bhat_i\cap \Bhat_j \neq \emptyset$, which is provided by Claim 1.
The assumption (i) actually implies that $|Ric|$ is bounded uniformly by $(n-1)e^\beta\alpha^{-1}e^{-1}$.
Hence, by the volume comparison theorem and the disjointness of $B_g(x_i, r_i/3)$, $i=1,2,...$, standard arguments will show that the maximal intersection number $I$ is finite.

\end{proof}

\section{Transfer rate of the Ricci flow}\label{section: transfer rate of the ricci flow}

Recall that the heat kernel on the Euclidean space $\mathbb{R}^n$ is
\[K(x,y,t)=\frac{1}{(4\pi t)^{n/2}}e^{-|x-y|^2/4t}.\]
If the initial data is compactly supported, the heat solution will decay spatially at the approximate rate of $e^{-|x|^2/4t}$ for positive $t$. As a Ricci flow analogue, we prove Theorem \ref{theorem: transfer rate estimate}, which is a consequence of Theorem \ref{theorem: existence of ricci flow on noncompact manifolds} and the curvature estimates we obtained during its proof.

\begin{proof}[Proof of Theorem \ref{theorem: transfer rate estimate}]

We shall verify that under the assumptions of the theorem, we can construct a good cover of the initial manifold that satisfies Assumption \ref{initial good cover assumption}. On $\Omega_0$ the construction is simple since we can cover it by finitely many balls with uniform radius and uniform $K_i$. Therefore we can assume WLOG that $\Omega_0 =\emptyset$.

For any $x \in M$, let's define
\[K_x = 64\bar{r}^2(1+d(x,p))^2, \quad r_x = \frac{\bar{r}}{\sqrt{K_x}}.\]
Then
\[
\begin{split}
\sup_{y \in B(x, 16 r_x)} |Rm|\leq & \sup_{y \in B(x, 16 r_x)} (1+ d(x,p) +d(x,y))^{2} \\
\leq & 2(1+d(x,p))^{2} + 2 (16r_x)^2\\
\leq & 32^{-1} K_x + 2 (\frac{16\bar{r}}{\sqrt{K_x}})^2,
\end{split}
\]
using $K_x \geq 64\bar{r}^2$, note that $\bar{r} > 1$, we get
\[
\sup_{B(x, 16 r_x)} |Rm|\leq  K_x/32 + K_x/8 < K_x.
\]
For a compact domain $\Omega \subset M$, let $\{x_i \in \Omega, i=1,2,...,N\}$ be a maximal set of points in $\Omega$ such that $B(x_i, \frac{1}{3} r_{x_i})$, $i=1,2,...,N$ are disjoint. For convenience, we denote $r_i = r_{x_i}$ and $B_i =B(x,r_i)$.

We claim that $\Omega \subset \cup_{i=1}^N B_i $.
For any pair of points $x,y$ with distance $d(x,y) \leq \frac{1}{4}$, the triangle inequality implies
\[
K_y \leq  64\bar{r}^2(1+d(x,p) + d(x,y))^2 \leq  K_x + 4\bar{r}\sqrt{K_x}+4\bar{r}^2 \leq 2K_x,
\]
hence
\[\frac{1}{\sqrt{2}} r_y \leq r_x \leq \sqrt{2} r_y. \]
If there is a point $y\in \Omega$, and $y \notin \cup_{i=1}^N B_i$. Then $d(x_i,y) > r_i$, for all $i=1,2,...,N$. Since $r_x \leq \frac{1}{8}$ for any $x\in M$, we know $B(y, \frac{1}{3} r_y) \cap B(x_i, \frac{1}{3} r_i) =\emptyset$ when $d(y, x_i) > \frac{1}{4}$. When $d(y, x_i) \leq \frac{1}{4}$, the above analysis shows $\frac{1}{3}r_i + \frac{1}{3}r_y \leq \frac{1+\sqrt{2}}{3} r_i < r_i \leq d(y, x_i)$, so again we have $B(y, \frac{1}{3} r_y) \cap B(x_i, \frac{1}{3} r_i) =\emptyset$. But this contradicts with the maximality of the set $\{x_i \in \Omega, i=1,2,...,N\}$.

Then by working on a compact exhaustion of $M$ and taking subsequential limits, we can construct a good cover of $M$, which satisfies (a) and (b) of Assumption \ref{initial good cover assumption}. It is not hard to observe that
\begin{equation}\label{equation: equivalence of K and distance squared}
\bar{r}^2(1+d(y,p))^2 \leq K_x \leq 2000\bar{r}^2 (1+d(y,p))^2, \quad when \hspace{4pt} d(x,y) \leq 16r_x.
\end{equation}
Hence $(c)$ can be verified by taking $A =\frac{\alpha}{2000\bar{r}^2}$. ($\bar{K}$ will depend on $\Omega_0$ eventually, we can take $\bar{K}=0$ when $\Omega_0=\emptyset$.)

Now let $B_i$ and $B_j$ be adjacent in the sense that $\Bhat_i \cap \Bhat_j \neq \emptyset$, where $\Bhat_i = B(x_i, 16r_i)$. Then
\[d(x_i, x_j) \leq 16 r_i + 16 r_j \leq 4,\]
hence $x_i$ and $x_j$ can be connected by a piecewise geodesic with $15$ pieces, and each piece has length less than $\frac{1}{4}$, therefore
\[d(x_i, x_j) \leq (2+\sqrt{2})^{15} r_i.\]
Then the triangle inequality implies
\[
\begin{split}
K_j \leq & 64\bar{r}^2(1+d(x_i,p) + d(x_i,x_j))^2 \\
\leq & K_i + 16\bar{r} \sqrt{K_i} d(x_i, x_j) + 64\bar{r}^2 d(x_i, x_j)^2 \\
\leq & K_i + 16 (2+\sqrt{2})^{15} \bar{r}^2 + 1024\bar{r}^2.
\end{split}
\]
Let $\Gamma = (16 (2+\sqrt{2})^{15} + 1024)\bar{r}^2$, then we have verified (d) in Assumption \ref{initial good cover assumption}.

Then using the disjointness of $B(x_i, \frac{1}{3}r_i)$, $i=1,2,...$, (e) is verified by applying the volume comparison theorem.

Therefore,  by Theorem \ref{theorem: existence of ricci flow on noncompact manifolds} and Lemma \ref{lemma: barrier estimates of Ric and its derivatives}, there is a solution of the Ricci flow starting with $(M,g)$, satisfying
\[
|\nabla^m Ric|^2(x,t) \leq \sum_{i=1}^\infty \phi_i(x) K_i^{2+m} e^{A\bar{K}+\Lambda K_i t - 2AK_i}, \quad m=0,1,2.
\]
for $(x,t) \in M \times [0,T_0]$, where $T_0 >0$ depends on $\Omega_0, n, \alpha$. Hence
\[
|\nabla^m Ric|^2(x,t) \leq C(n,\Omega_0, \alpha) K_i^{2+m} e^{ - AK_i}, \quad m=0,1,2.
\]
when $0< t \leq \min \{T_0, \frac{A}{\Lambda}\}:=T_1$ and $x\in B_i$.

The proof can be finished  by using (\ref{equation: equivalence of K and distance squared}) and
choosing suitable constants $C_1$ and $C_2$.
\end{proof}

\bibliographystyle{plain}
\bibliography{ref}
\end{document}